\documentclass[12pt]{compositio}
\usepackage[utf8]{inputenc}
\usepackage{latexsym}
\usepackage{amscd, amsfonts, eucal, mathrsfs, amsmath, amssymb, amsthm} 
\input xy
\xyoption{all}

\usepackage{pdfsync}
\usepackage{newcent}
\usepackage{comment}
\usepackage{tikz-cd}
\usepackage{extarrows}
\tikzset{labl/.style={anchor=south, rotate=90, inner sep=.5mm}}
\tikzset{labr/.style={anchor=north, rotate=90, inner sep=1mm}}

\usepackage{longtable}

\usepackage{rotating}
\usepackage{graphicx}

\usepackage{listings}
\usepackage{longtable}

\usepackage[cal=boondoxo, scr=boondoxo]{mathalfa}

\renewcommand{\arraystretch}{1.5}

\usepackage{hyperref}

\usepackage{rotating}

\newcommand{\bPhi}{\underline{\Phi}}

\DeclareMathOperator{\Quot}{Quot}

\newcommand{\field}[1]{\mathbf #1}
\newcommand{\mf}[1]{\mathfrak #1}

\newcommand{\ms}[1]{\mathcal #1}

\renewcommand{\phi}{\varphi}

\newcommand{\FT}{\Delta^+}

\newcommand{\aff}{\text{\rm aff}}

\DeclareMathOperator{\Poly}{\ms P}

\DeclareMathOperator{\ord}{ord}

\DeclareMathOperator{\Proj}{Proj} \DeclareMathOperator{\Sym}{Sym}

\newcommand{\WeierstrassFibrations}{\mathcal{W\!F}}
\newcommand{\WeierstrassData}{\mathcal{WD}}

\DeclareMathOperator{\I}{I} \DeclareMathOperator{\II}{II}
\DeclareMathOperator{\III}{III} \DeclareMathOperator{\IV}{IV}

 \DeclareMathOperator{\codim}{codim}

\newcommand{\R}{\field R}

\newcommand{\C}{\field C}
\newcommand{\F}{\field F}
\newcommand{\Z}{\field Z}

\newcommand{\N}{\field N}

\newcommand{\simto}{\stackrel{\sim}{\to}}

\DeclareMathOperator{\Spec}{Spec} \DeclareMathOperator{\spec}{Spec}

\renewcommand{\P}{\field P}

\newcommand{\A}{\field A}

\DeclareMathOperator{\Pic}{Pic}

\DeclareMathOperator{\GL}{GL} \DeclareMathOperator{\PGL}{PGL}

\DeclareMathOperator{\NS}{NS}

\newcommand{\m}{\boldsymbol{\mu}}

\newcommand{\G}{\field G} %for the multiplicative and additive groups

\renewcommand{\H}{\operatorname{H}}

\newcommand{\llbracket}{[\!\hspace{0.03em}[}
\newcommand{\rrbracket}{]\!\hspace{0.03em}]}
\newcommand{\llparen}{(\!\hspace{0.03em}(}
\newcommand{\rrparen}{)\!\hspace{0.03em})}

\DeclareMathOperator*{\tensor}{\otimes}

\newcommand{\inj}{\hookrightarrow}

\newcommand{\id}{\operatorname{id}}

\DeclareMathOperator{\Aut}{\operatorname{Aut}}

\DeclareMathOperator{\Isom}{\operatorname{Isom}}
\DeclareMathOperator{\Hom}{\operatorname{Hom}}

\DeclareMathOperator{\M}{\operatorname{M}}

\DeclareMathOperator{\Br}{\operatorname{Br}}

\newcommand{\eps}{\varepsilon}

\renewcommand{\mathbb}{\mathbf}

\newtheorem{lem}{Lemma}[section]

\newtheorem{thm}[lem]{Theorem}

\newtheorem{prop}[lem]{Proposition}

\newtheorem{cor}[lem]{Corollary}

\theoremstyle{definition}
\newtheorem{defn}[lem]{Definition}

\newtheorem{example}[lem]{Example}

\newtheorem{notation}[lem]{Notation}

\newtheorem{conditions}[lem]{Conditions}

\theoremstyle{remark}
\newtheorem{remark}[lem]{Remark}

\newtheorem{question}[lem]{Question}

\numberwithin{equation}{lem}

\author{Daniel Bragg}
\email{braggdan@berkeley.edu}
\address{Evans Hall, University of California, Berkeley CA 94720, USA}

\author{Max Lieblich}
\email{lieblich@uw.edu}
\address{Padelford Hall, University of Washington, Seattle WA 98195, USA}

\title[Perfect points on curves of genus one]{Perfect points on curves of genus one and consequences for supersingular K3 surfaces}


\begin{document}
\begin{abstract}
  We describe a method to show that certain elliptic surfaces do not admit
  purely inseparable multisections (equivalently, that genus one curves over
  function fields admit no points over the perfect closure of the base field)
  and use it to show that any non-Jacobian elliptic structure on a very general
  supersingular K3 surface has no purely inseparable multisections. We also
  describe specific examples of genus 1 fibrations on supersingular K3 surfaces without purely inseparable
  multisections.
  %Finally, we discuss the consequences for the claimed proof of 
  %the Artin conjecture on unirationality of supersingular K3 surfaces.
\end{abstract}
\maketitle
\setcounter{tocdepth}{2}
\tableofcontents

\section{Introduction}
\label{sec:introduction}

In this note we study the following question. Fix an algebraically closed field
$k$ of characteristic at least $5$.

\begin{question}\label{ques:question}
  When does an elliptic surface $f:X\to\P^1$ admit a purely inseparable
  multisection? Equivalently, when does the corresponding genus $1$ curve
  $C/k(t)$ have points over the perfect closure $k(t)^{\text{\rm perf}}$?
\end{question}

There has been some previous work on purely inseparable points on curves
\cite{MR1484697}, abelian varieties \cite{MR3317331,1702.07142}, and torsors
\cite{1808.03960}. Roughly speaking, these authors have found that for
non-isotrivial fibrations the set of purely inseparable points tends to be small.
Thus, one might suspect that it is quite often the case that this set is empty.

We will give a complete answer to Question \ref{ques:question} below for $X$
that are suitably general in the moduli space of supersingular K3 surfaces. One
of the main consequences of our techniques is the following.

\begin{thm}\label{thm:very general bidness}
  Suppose $i$ is $8$, $9$, or $10$. For a very general supersingular K3 surface
  $X$ of Artin invariant $i$, every elliptic fibration $f:X\to\P^1$ has the
  property that it has a purely inseparable multisection if and only if it has a
  section.
\end{thm}
That is, for genus $1$ curves $C/k(t)$ whose minimal model is a suitably general
supersingular K3 surface, the presence of a purely inseparable point implies the existence of a rational point.

If $X$ is a supersingular K3 surface of Artin invariant $10$, then no elliptic fibration on $X$ has a section (see e.g. \cite[Proposition 13.1]{MR3300417}). Thus, Theorem \ref{thm:very general bidness} implies the following result.

\begin{cor}
    If $X$ is a very general supersingular K3 surface, then no elliptic fibration $f:X\to\P^1$ has a purely inseparable multisection. 
\end{cor}

\subsection{The idea}
The basic idea of the proof is the following. We can cover the moduli space of
supersingular K3 surfaces by Artin--Tate families (as described in
\cite{1804.07282}), so, by restricting to the first-order deformations of
universal formal Brauer classes, we can reduce the existence of purely
inseparable multisections on elliptic fibrations on general supersingular K3
surfaces to a problem about the action of powers of Frobenius on coherent
cohomology. We briefly go into more detail and give a simultaneous outline of
the paper.

Let $f:X\to\P^1$ be a non-Jacobian elliptic supersingular K3 surface. The fibration $f:X\to\P^1$ can be thought of as a torsor under the
Jacobian fibration $g:J\to\P^1$. Since $X$ is a supersingular K3 surface, this torsor must
have index $p$. The Artin--Tate isomorphism \cite[\S 3]{MR1610977} tells us that
$X$ corresponds to an element $\alpha$ of the Brauer group of $J$. Moreover, by
the special properties of supersingular K3 surfaces, the fibration $f$ can be
put in a $1$-parameter family of $J$-torsors that contains $J$ itself. That is,
$\alpha$ fits into a family of Brauer classes
$\widetilde\alpha\in\Br(J\times\A^1)$ whose restriction
$\widetilde\alpha|_{t=0}$ is trivial. (These are called \emph{Artin--Tate
families\/} and are a particular case of the \emph{twistor spaces\/} studied in
\cite{1804.07282}.) The tangent value of $\alpha$ over $k[t]/(t^2)$ can
naturally be thought of as an element of $\H^2(J,\ms O_J)$. 

As explained in Section \ref{sec:main event}, if $X$ is to have a purely
inseparable multisection, we must have that $\alpha$ vanishes after passing to a
minimal model $Y$ of the pullback of $J\to\P^1$ by some power of the Frobenius
morphism $F:\P^1\to\P^1$. The action of the Frobenius on the restriction
$\widetilde\alpha|_{k[t]/(t^2)}$ is identified with the pullback action
$\H^2(J,\ms O_J)\to\H^2(Y,\ms O_Y)$. We are thus led to try to understand when
this pullback map is injective, for if it is, then a general restriction of
$\widetilde\alpha$ cannot become zero upon such a pullback. (This last statement
involves an understanding of the way in which the Artin--Tate families trace out
the moduli space of supersingular K3 surfaces.)

As we discuss in Section \ref{sec:frobenius split elliptic fibrations} (using
technical inputs about divisors explained in Section \ref{sec:frob-split}), the
properties of the Frobenius pullback map on cohomology are determined in large
part by the configuration of singular fibers of $g$. Combining Lemma \ref{lem:frob split}
and Proposition \ref{prop:cohoinj}, we describe a condition on the singular fibers of an
elliptic fibration which implies that this pullback map is injective. Our condition can
be easily checked given a Weierstrass model for the elliptic surface. In Section
\ref{sec:explicit examples}, we record some explicit examples of elliptic
supersingular K3 surfaces verifying this condition and hence deduce the
existence of elliptic supersingular K3 surfaces without purely inseparable
multisections (see Corollary \ref{cor:twistor}ff and Corollary \ref{cor:they
exist!}).

We next wish to demonstrate the existence of supersingular K3 surfaces that have
no (non-Jacobian) elliptic fibrations with purely inseparable multisections.
(Any supersingular K3 surface admits many elliptic structures, so this is a
stronger claim). Consider a Jacobian elliptic fibration, with canonically
associated Weierstrass model given by an equation $y^2=x^3+a(t)x+b(t)$. The
singular fibers of additive reduction (which determine the behavior of the
Frobenius pullback) impose conditions on the polynomials $a$ and $b$. One can
thus compute (or at least bound the dimension of) the locus of Weierstrass
equations where the pullback by powers of the Frobenius acts trivially on $\H^2(J,\ms O_J)$. On
the other hand, one also knows the dimension of the locus of Weierstrass
equations corresponding to supersingular K3 surfaces of fixed Artin invariant.
Putting these dimension bounds together, we can show that the Artin invariant
$7$ (or higher) locus in the space of Weierstrass equations is too large to be
contained in the locus of Weierstrass equations for which some power of
Frobenius is non-injective on $\H^2$. We carry this out in Section
\ref{sec:moduli} and Section \ref{sec:non-pi-inj}. In Section \ref{sec:very
general}, we show how the theory of twistor lines developed in \cite{1804.07282}
then gives Theorem \ref{thm:very general bidness}.

Since writing this paper, we have learned that Fakhruddin seems to be the first
person to observe that one can create a geometric realization of the universal
formal Brauer element of a K3 surface as a deformation of a Jacobian elliptic
fibration. The reader is referred to \cite[Lemma 4]{MR1904084} for details. We
thank Fakhruddin for making us aware of his work.

\subsection{Consequences for the Artin conjecture}
  In \cite{Liedtke15}, the existence of purely inseparable multisections for
  generic members of Artin--Tate families (in the notation of \cite{1804.07282})
  is asserted as the ``technical heart" \cite[Proposition 3.5 and Theorem
  3.6]{Liedtke15} of the work on Artin's conjecture that supersingular K3 surfaces are unirational. Theorem
  \ref{thm:very general bidness} shows that, in fact, such purely inseparable
  multisections almost never exist, in the sense that a very general
  supersingular K3 surface admits no fibrations that admit such multisections.
  We also include explicit examples in Section \ref{sec:explicit examples} that
  directly contradict \cite[Proposition 3.5]{Liedtke15} (see Corollary
  \ref{cor:twistor}ff and Corollary \ref{cor:they exist!}). The examples in Section \ref{sec:explicit examples} and the more general result in Theorem \ref{thm:very general bidness} invalidate
  \cite[Theorem 5.1 and Theorem 5.3]{Liedtke15} (cf. \cite[p 981]{Liedtke15}).
  
  The errors in \cite{Liedtke15} have now been acknowledged in \cite{Liedtke2021}. In particular, the Artin conjecture remains a conjecture.

\subsection*{Acknowledgments}

During the course of this work, we received helpful comments from Tony
Várilly-Alvarado, Jennifer Berg, Bhargav Bhatt, Daniel Hast, Wei Ho, Daniel
Huybrechts, Christian Liedtke, Davesh Maulik, Yuya Matsumoto, Martin Olsson, and
Matthias Schütt. We also received numerous helpful comments from the referee.

The second author thanks Rice University, Stanford University, and the
University of California at Berkeley for their hospitality during the 2018--2019
academic year, when this work was carried out. He was also supported by NSF
grant DMS-1600813 and a Simons Foundation Fellowship during this period. This
material is based in part upon work supported by the National Science Foundation
under Grant Number DMS-1440140 while the authors were in residence at the
Mathematical Sciences Research Institute in Berkeley, California, during the
Spring 2019 semester. 

\section{Definitions and notation}\label{sec:defs}

We work over a fixed algebraically closed field $k$ of positive characteristic $p$. We assume throughout that $p\geq 5$. This assumption is made to avoid the more complicated behavior of singular fibers of elliptic surfaces in characteristics $2$ and $3$.

Let $X$ be a smooth proper surface over $k$. An \textit{elliptic fibration} on
$X$ is a flat proper generically smooth morphism $X\to C$ to a smooth proper
curve whose generic fiber has genus $1$. An \textit{elliptic surface} is a smooth proper surface $X$ equipped with
an elliptic fibration.

An elliptic fibration is \textit{Jacobian} if it admits a section and
\textit{non-Jacobian} otherwise.

A \textit{multisection} of an elliptic fibration is an integral subscheme
$\Sigma\subset X$ such that the map $\Sigma\to C$ is finite and flat. A
multisection $\Sigma$ is \textit{purely inseparable} if $\Sigma\to C$ induces a
purely inseparable map on function fields.

\section{Frobenius-split divisors}
\label{sec:frob-split}

In this section we will study a certain condition on divisors $D\subset\P^1$. We
will use this in the following sections to study relative Frobenius splittings
of elliptic K3 surfaces.

We consider $\P^1_{\Z}:=\Proj\Z[s,t]$ equipped with its \emph{canonical\/}
system of homogeneous coordinates. Given a positive integer $m$, let
$\rho^{(m)}:\P^1\to\P^1$ denote the morphism given in homogeneous coordinates by
$[s:t]\mapsto [s^m:t^m]$. 
\begin{defn}\label{defn:split}
	A divisor $D\in|\ms O_{\P^1}(n)|$ is \emph{$m$-split\/} if the morphism
	$$\ms O_{\P^1}\to\rho^{(m)}_\ast\ms O_{\P^1}(n)$$ adjoint to the map
	$$(\rho^{(m)})^\ast\ms O_{\P^1}=\ms O_{\P^1}\to\ms O_{\P^1}(n)$$ associated to
	$D$ admits a splitting (as $\ms O_{\P^1}$-modules).
\end{defn}
\begin{remark}
	Since $k$ is algebraically closed, a divisor is $p^e$-split if and only if the
	natural morphism $\ms O_{\P^1}\to F^e_\ast\ms O_{\P^1}(n)$ is split, where $F$
	is the \emph{absolute\/} Frobenius. We will use this observation implicitly in
	what follows.
\end{remark}

Given a divisor $D\in|\ms O(n)|$, there is a naturally associated homogeneous
form $f(s,t)$ of degree $n$ whose vanishing locus is $D$; $f$ is unique up to
scaling.

\begin{prop}\label{prop:frob-split}
	Let $D\in |\ms O(n)|$ be a divisor with associated homogeneous form $f(s,t)$.
	Fix a positive integer $m$.
	\begin{enumerate}
		\item If $0\leq n<m$ then any divisor $D$ is $m$-split.\label{item:low power}
		\item If $m\leq n<2m$ then the divisor $D$ is $m$-split if and only if the
		form $f$ cannot be written as $s^mg+t^mh$ for homogeneous forms $g$ and $h$
		of degree $n-m$. (In other words, some term in $f$ is not divisble by $s^m$
		or $t^m$.)
	\end{enumerate}
\end{prop}
\begin{proof}
	Write $n=m+\delta$ and $\rho$ for $\rho^{(m)}$. We have
	$$\rho_\ast\ms O(\delta)=\ms O^{\delta+1}\oplus\ms O(-1)^{m-\delta-1}$$ This
	is shown for instance in the introduction to \cite{MR2968908} when $\rho$ is a
	power of the Frobenius, but the proof given applies unchanged in our
	situation. In terms of graded modules over $k[s,t]$, the corresponding graded
	module
	$$M_\bullet=\bigoplus_n\H^0(\P^1,(\rho_\ast\ms O(\delta))(n))$$ admits the
	following description: $M_n$ is the free $k$-vector space spanned by $s^it^j$
	with $i+j=\delta+mn$. Given positive integers $i$ and $j$ such that
	$i+j=\delta$, let $N^{i,j}$ be the graded submodule of $M_\bullet$ with
	$N^{i,j}_n$ the submodule spanned by $s^at^b$ such that $a+b=\delta+mn$,
	$a\equiv i\pmod{m}$, and $b\equiv j\pmod{m}$. There is a graded splitting 
	$$M_\bullet\to N^{i,j}_\bullet$$ defined as follows: given $\alpha$ and
	$\beta$ such that $\alpha+\beta=\delta+mn$, send $s^\alpha t^\beta$ to $0$ if
	$\alpha\not\equiv i\pmod m$ or $\beta\not\equiv j\pmod m$, and to $x^\alpha
	y^\beta$ otherwise. There results a free summand
	$$N_\bullet=\bigoplus_{i+j=a}N^{i,j}_\bullet\subset M_\bullet.$$ Since the
	rank of $N$ is $\delta$, we see that this gives the full free summand of
	$\rho_\ast\ms O(\delta)$.
	
	To compute the sections of $\rho_\ast\ms O(m+\delta)=\rho_\ast\ms O(n)$, we
	shift the decomposition $M_\bullet=N_\bullet\oplus K_\bullet$ by $1$. That is,
	the non-split divisors will correspond to the elements of $N_\bullet$ of
	degree $1$. These are precisely the forms of degree $n$ in $s$ and $t$ each of
	whose terms are divisible by $s^m$ or $t^m$. This gives the desired result.
\end{proof}
\begin{remark}
	Geometrically, the $m$-split divisors in $|\ms O(n)|$ are given by the linear
	span of the copy of $\P^1\times|\ms O(n-m)|\subset|\ms O(n)|$ corresponding to
	divisor sums $\Phi+E$ with $\Phi$ a fiber of $\rho$ and $E$ a member of $|\ms
	O(n-m)|$.  
\end{remark}

We end this section with a crude estimate on the dimension of the $m$-split
locus in a particular space of polynomials. Suppose given a non-increasing
sequence $(\lambda_1,\ldots,\lambda_n)$ of positive integers such that
$\sum\lambda_i=m+\delta$ with $0\leq\delta<m$. Consider the polynomials
$$P_{\lambda_1,\ldots,\lambda_n}(t,s):=\prod(t-z_is)^{\lambda_i}$$ parametrized
by the affine space $\A^n$ with coordinates $z_1,\ldots,z_n$. When the
$\lambda_i$ and $z_j$ are clear, we will write simply $P$. The polynomial $P$
corresponds to the divisor
$$D_{\lambda_1,\ldots,\lambda_n}(z_1,\ldots,z_n):=\sum\lambda_i z_i$$ in $\P^1$.
Write $\tau_j(P)=\sum_{i=1}^j\lambda_i$ for the partial sums of the $\lambda_i$.
When $P$ is understood, we will write simply $\tau_j$. There results a strictly
increasing sequence $(\tau_1,\ldots,\tau_n)$. Let 
\begin{equation}\label{eq:the number B}
  B_m(\lambda_1,\ldots,\lambda_n)=\#\{i | \delta+1\leq\tau_i\leq m-1\}
\end{equation}

\begin{prop}\label{prop:frob-split-crude}
	With the above notation, the set $A$ of $D(z_1,\ldots,z_n)$ that are not
	$m$-split has codimension at least $B_m(\lambda_1,\ldots,\lambda_m)$ in
	$\A^n$.
\end{prop}
\begin{proof}
	First, we note that the coefficients of $P_{\lambda_1,\ldots,\lambda_n}$ are
	homogeneous in $z_1,\ldots,z_n$, so it suffices to prove the statement for the
	image of $A$ in the projective space $\P^{n-1}$ with homogeneous coordinates
	$z_1,\ldots,z_n$. Let $j$ be minimal such that $\tau_j\in[\delta+1,m-1]$.
	Consider the projective subspace $L\subset\P^{n-1}$ given by the vanishing of
	$z_1,\ldots,z_j$. This is a codimension $j$ subspace, corresponding to the
	polynomials 
	$$P_{\lambda_1+\cdots+\lambda_j,\lambda_{j+1},\ldots,\lambda_n}(s,t)=t^{\tau_j}(t-z_{j+1}s)^{\lambda_{j+1}}\cdots(t-z_ns)^{\lambda_n}$$
	in the variables $z_{j+1},\ldots,z_n$. The lowest $t$-degree term of this
	polynomial is 
	$$z_{j+1}^{\lambda_{j+1}}\cdots
	z_n^{\lambda_n}t^{\tau_j}s^{p+\delta-\tau_j}.$$ Since $\tau_n<2m$, we see that
	this term must vanish in order for the polynomial not to be $m$-split. This
	means that one of $z_{j+1},\cdots,z_n$ must vanish, giving a union of
	hyperplanes in $L$. The restriction of $P$ to each of these hyperplanes has
	the form 
	$$P_{\tau_j+\lambda_q,\lambda_{j+1},\ldots,\lambda_n}$$ (with $\lambda_q$
	taken out of the sequence). Since the $\lambda_i$ are non-increasing, the
	shortest sequence of such polynomials whose lowest $t$-degree term lies in
	$[\delta+1,m-1]$ has length $B:=B_m(\lambda_1,\ldots,\lambda_m)$. 
	It follows that $A\cap L$ has codimension at
	least $B$ in $L$. Since the codimension can only go down upon intersection, we conclude
	that $A$ has codimension at least $B$, as desired.
\end{proof}

\begin{cor}\label{cor:split powers}
	In the above notation, suppose $\sum\lambda_i\leq 24$. For any positive
	integer $q$ we have
  $$B_{12q+1}(q\lambda_1,\ldots,q\lambda_n)=B_{13}(\lambda_1,\ldots,\lambda_n).$$
\end{cor}
\begin{proof}
  Write $\sum\lambda_i=13+\delta$ and $\sum q\lambda_i=12q+1+\delta'$. We have
  $13q+\delta q=12q+1+\delta'$, so that $q(1+\delta)=\delta'+1$. Note that the
  partial sums $\tau_j(q\lambda_1,\ldots,q\lambda_n)$ satisfy 
  $$\tau_j(q\lambda_1,\ldots,q\lambda_n)=q\tau_j(\lambda_1,\ldots\lambda_n).$$
  The bounds on the partial sums $\tau_j(q\lambda_1,\ldots,q\lambda_n)$ to
  compute $B_{12q+1}(q\lambda_1,\ldots,q\lambda_n)$ are $\delta'+1$ and $12q$.
  On the other hand, by the above calculation we have $\delta'+1=q(1+\delta)$.
  The set of $\tau_j(q\lambda_1,\ldots,q\lambda_n)$ lying between $\delta'+1$
  and $12q$ is thus the same as the set of $\tau_j(\lambda_1,\ldots,\lambda_n)$
  lying between $\delta+1$ and $12$. This gives the desired equality.
\end{proof}

\begin{cor}\label{cor:codim-case-13}
  Suppose $\sum_{i=1}^n\lambda_i\leq 24$ and let
  $B=B_{13}(\lambda_1,\ldots,\lambda_n)$. Then the locus of
  $(z_1,\ldots,z_n)\in\A^n$ such that 
	$$\frac{p^{2e}-1}{12}D(z_1,\ldots,z_n)$$ is not $p^{2e}$-split for some
	positive integer $e$ can be written as an ascending union of closed subschemes
	$Z_{2e}$ each of codimension at least $B$. In particular, $\bigcup_e Z_{2e}$
	cannot contain any closed subspace $Y\subset D(z_1,\ldots,z_n)$ of codimension
	smaller than $B$.
\end{cor}
\begin{proof}
  This follows immediately from Corollary \ref{cor:split powers}, once we note
  that $p^{2e}\equiv 1\pmod{12}$ for all positive integers $e$ and all primes
  $p>3$.
\end{proof}

\begin{example}\label{example:8II}
	Consider the polynomial
	$$(t-z_1s)^2(t-z_2s)^2\cdots(t-z_8s)^2$$ of total degree $16$. The partial
	sums are $(2,4,6,8,10,12,14,16)$. We have
	\[
	  B_{13}(2,2,2,2,2,2,2,2)=\#\left\{i|4\leq\tau_i\leq 12\right\}=5
	\]
	We conclude that the codimension of the non-$13$-split locus in $\A^8$ is at
	least $5$. As we will see below, this implies that the locus of Weierstrass
	fibrations associated to elliptic K3 surfaces whose additive fiber
	configuration is $8\II$ and which are not $\infty$-Frobenius split (see
	Section \ref{sec:frobenius split elliptic fibrations}) has codimension at
	least $13$ in the space of Weierstrass fibrations (see Section \ref{sec:moduli}). Note that the locus admitting configuration $8\II$ has codimension 8 in the space of Weierstrass fibrations, as explained in Section \ref{sec:moduli}. This is the smallest
	codimension bound on a non-$\infty$-Frobenius-split locus to come out of our methods.
\end{example}

\section{Frobenius split elliptic fibrations}\label{sec:frobenius split elliptic fibrations}

In this section we introduce a notion of Frobenius splitting of an elliptic
fibration. Let $f:X\to C$ be an elliptic surface. For any $e>0$, there is an
induced diagram
\begin{equation}\label{eq:the situation room}
\begin{tikzcd}
Y_e\ar[dr, "\phi_e"'] \ar[r, "g_e"] & X^{(e)}\ar[r, "W^e"]\ar[d, "f^{(e)}"] & X\ar[d, "f"] \\
& C\arrow{r}{F^e} & C
\end{tikzcd}
\end{equation}
where $F^e$ is the $e$th power of the absolute Frobenius of $C$, the square is
cartesian, and $g_e:Y_e\to X^{(e)}$ is the minimal resolution of $X^{(e)}$. For
the sake of symmetry, we will write $Y_0=X$ and $\phi_0=f$. There is an induced
map
\begin{equation}\label{eq:cohomap11}
\R^1f_*\ms O_X\to F^e_*\R^1(\varphi_e)_*\ms O_{Y_e}
\end{equation}
of locally free sheaves on $C$. We note that $\R^1f_*\ms O_X$ is an invertible
sheaf, 
%(dual to the fundamental line bundle $\mathbb{L}$ of \cite[Section II.4]{MR1078016}), 
while the sheaf on the right hand side is locally free of
rank $p^e$.
\begin{defn}\label{def:frobenius split elliptic fibration}
	An elliptic fibration $f:X\to C$ is \emph{$e$-Frobenius-split\/} if the map
	\eqref{eq:cohomap11} is split. If $f$ is $e$-Frobenius-split for all $e>0$
	then we say that $f$ is \emph{$\infty$-Frobenius-split\/}.  
\end{defn}

Note that both sheaves in (\ref{eq:cohomap11}) are unchanged upon replacing $X$
and $Y_e$ with their relatively minimal models. In particular, if $f':X'\to C$ is
the relatively minimal model of $f$, then $f'$ is $e$-Frobenius split if and only
if $f$ is $e$-Frobenius split.

Recall that a variety $Z$ is $e$-Frobenius split if the map $\ms O_Z\to
F^e_{Z*}\ms O_Z$ is split.
\begin{prop}
	Let $f:X\to C$ be an elliptic surface and $e$ a positive integer.
	\begin{enumerate}
		\item If $X$ is $e$-Frobenius split, then $f$ is $e$-Frobenius split.
		\item If $f$ is $e$-Frobenius split, then $C$ is $e$-Frobenius split.
	\end{enumerate}
\end{prop}
\begin{proof}
	The $e$th absolute Frobenius $F_X^e\colon X\to X$ factors through $X^{(e)}$.
	By minimality of $Y_e$, it also factors through the map $g_e$, giving a map
	$h:X\to Y_e$, and hence a map of sheaves $\ms O_{Y_e}\to h_*\ms O_X$. Applying
	$F^e_*\R^1(\varphi_e)_*$, we find a map
	$$
	F^e_*\R^1(\varphi_e)_*\ms O_{Y_e}\to F^e_*\R^1(\varphi_e)_*h_*\ms O_X\to\R^1f_*F^e_{X*}\ms O_X
	$$
	where the second map is induced by the appropriate Leray spectral sequence. We
	obtain a diagram
	\begin{equation}\label{eq:diagram!!!}
	\begin{tikzcd}
	\R^1f_*\ms O_X\arrow{d}\arrow{dr}\arrow[bend left=10]{drr}&&\\
	F^e_*\R^1f^{(e)}_*\ms O_{X^{(e)}}\arrow{r}&F^e_*\R^1(\varphi_e)_*\ms
	O_{Y_e}\arrow{r}&\R^1f_*F^e_{X*}\ms O_X
	\end{tikzcd}
	\end{equation}
	Note that the vertical arrow is identified with the canonical map $\ms O_C\to
	F^e_*\ms O_C$ tensored with the invertible sheaf $\R^1 f_*\ms O_X$. Now, a
	splitting of $\ms O_X\to F^e_*\ms O_X$ gives a splitting of the long diagonal
	arrow of \eqref{eq:diagram!!!}, which induces a splitting of the diagonal
	arrow \eqref{eq:cohomap11}, and a splitting of \eqref{eq:cohomap11} induces a
	splitting of the vertical arrow. 
\end{proof}

\begin{remark}
	Neither the converse of (i) nor the converse of (ii) holds (see Examples
	\ref{ex:shioda example1} and \ref{ex:shioda example3}). Indeed, we will show
	that if $X$ is a very general supersingular K3 surface of Artin invariant
	$\leq 9$, then every Jacobian elliptic fibration on $X$ is $e$-Frobenius split
	for all $e$. However, a supersingular K3 surface is never $e$-Frobenius split for any $e$.
\end{remark}

\begin{example}\label{ex:semistable implies split}
	Suppose $f:X\to C$ is an elliptic surface that is semistable (that is, $f$ has
	no additive fibers). This implies, with the notation of the diagram
	(\ref{eq:the situation room}), that the Frobenius pullback $X^{(e)}$ has only
	rational singularities. By construction, the map (\ref{eq:cohomap11}) always
	factors as
	\begin{equation}\label{eq:composition of some things}
	\R^1f_*\ms O_X\to F^e_*\R^1 f^{(e)}_*\ms O_{X^{(e)}}\xrightarrow{\sim} F^e_*\R^1 f^{(e)}_* g_{e*}\ms O_{Y_e}\to F^e_*\R^1(\varphi_{e})_*\ms O_{Y_e}
	\end{equation}
	As $X^{(e)}$ has rational singularities, the map $\R^1 f^{(e)}_* g_{e*}\ms
	O_{Y_e}\to \R^1(\varphi_{e})_*\ms O_{Y_e}$ is an isomorphism, and hence the
	right most arrow of (\ref{eq:composition of some things}) is an isomorphism.
	But the map
	$$
	\R^1f_*\ms O_X\to F^e_*\R^1 f^{(e)}_*\ms O_{X^{(e)}}
	$$
	is identified with the map
	$$
	\R^1f_*f^*(\ms O_C)\to \R^1f_*f^*(F^e_*\ms O_C)
	$$
	We conclude that a semistable elliptic surface over an $e$-Frobenius split
	curve is itself $e$-Frobenius split. In particular, a semistable elliptic K3
	surface is $\infty$-Frobenius split. We will strengthen this observation in
	Proposition \ref{prop:cohoinj2}.
\end{example}

\begin{lem}\label{lem:level lowering}
	If $f\colon X\to C$ is $e$-Frobenius split for some $e\geq1$, then $f$ is also
	$e'$-Frobenius split for all $1\leq e'\leq e$.
\end{lem}
\begin{proof}
	Using the universal property of the minimal resolution, we find a
	factorization
	$$
	  \R^1f_*\ms O_X\to\R^1(\varphi_{e'})_*\ms O_{Y_{e'}}\to\R^1(\varphi_e)_*\ms O_{Y_e}
	$$
	Thus, a splitting of (\ref{eq:cohomap11}) for $e$ gives rise to a splitting
	for $e'$.
\end{proof}

We record a cohomological consequence of being $e$-Frobenius split. Consider the
semilinear morphism
\begin{equation}\label{eq:cohomap}
\gamma_e(f):\H^2(X,\ms O_X)\to\H^2(Y_e,\ms O_{Y_e})    
\end{equation}
arising from the composition of $g_e^\ast$ with the canonical isomorphism
$$\H^2(Y_e,\ms O_{Y_e})\simto\H^2(X',\ms O_{X'})$$
\begin{lem}\label{lem:frob split}
	$f:X\to\P^1$ is $e$-Frobenius split then the map $\gamma_e(f)$
	(\ref{eq:cohomap}) is injective.
\end{lem}
\begin{proof}
	We have a commutative diagram
	$$
	\begin{tikzcd}
	\H^1(C,\R^1f_*\ms O_X)\arrow{r}\arrow{d}&\H^1(C,\R^1(\varphi_{e})_*\ms O_{Y_{e}})\arrow{d}\\
	\H^2(X,\ms O_X)\arrow{r}{\gamma_{e}(f)}&\H^2(Y_{e},\ms O_{Y_{e}})
	\end{tikzcd}
	$$
	where the vertical arrows are the isomorphisms induced by the Leray spectral
	sequence, and the upper horizontal arrow is the map on $\H^1$ induced by the
	map (\ref{eq:cohomap11}). A split injection of sheaves induces an injection on
	cohomology, which gives the result.
\end{proof}

\subsection{Weierstrass fibrations and Tate's algorithm}

We briefly recall how to determine the fiber
type given the (local) Weierstrass equation for an elliptic fibration. Suppose
$R$ is a dvr with uniformizer $t$, fraction field $K$, and algebraically closed
residue field $\kappa$ of characteristic at least $5$. Given an elliptic curve $X_K$
over $K$, there is a minimal Weierstrass equation
\begin{equation}\label{eq:minimal weierstrass equation}
    y^2=x^3+gt^\alpha x+ht^\beta,
\end{equation}
where $g$ and $h$ are units of $R$ and
$\alpha,\beta\in\N\cup\{\infty\}$, with the convention that $t^\infty=0$. Minimality is equivalent to the assertion that $\alpha<4$ or $\beta<6$ (or both). The discriminant of the curve
is given by the formula 
\[
    \Delta=4g^3t^{3\alpha}+27h^2t^{2\beta}
\]
We let $\delta$ denote the $t$-adic valuation of $\Delta$.

Tate's algorithm \cite{MR0393039} implies that the quantities $\alpha,\beta,$ and $\delta$ determine the Kodaira type $\Phi$ of the special fiber. We record the relations between the basic data $\alpha,\beta,\delta$ and the Kodaira type $\Phi$ in the first three rows of Table \ref{table:1}. In particular, we highlight the following consequences:
\begin{enumerate}
    \item the special fiber is smooth if and only if $\delta=0$;
    \item the special fiber is semistable if and only if either $\alpha=0$ or $\beta=0$ and is additive otherwise;
	\item if the fiber is additive, then unless $\alpha=2$ and $\beta=3$, the fiber type $\Phi$ is determined by $\delta$;
	\item if $\alpha=2$ and $\beta=3$, the fiber has type $\I_n^\ast$ where	$n=\delta-6$.
\end{enumerate}

%If we assume a certain value $\Phi$ for the Kodaira type of the special fiber, Tate's algorithm gives restrictions on the possible values of the basic data $\alpha,\beta,\delta$. 

\renewcommand{\arraystretch}{1.1}

\begin{table}[h]
  \centering
  \begin{tabular}{|c||c|c|c|c|c|c|c|c|c|c|}
    \hline
    $\Phi$ & $\I_0$ & $\I_{n\,\,(\substack{n\geq 1})}$ & $\II$ & $\III$ & $\IV$ & $\I_0^\ast$ & $\I^\ast_{n\,\,(\substack{n\geq 1})}$ & $\IV^\ast$ & $\III^\ast$ &
    $\II^\ast$ \\
    \hline
    $\alpha$  & $\geq 0$ & $0$ & $\geq 1$ & $1$ & $\geq 2$ & $\geq 2$ & $2$ & $\geq 3$ & $3$ & $\geq 4$ \\
    \hline
    $\beta$   & $\geq 0$ & $0$ & $1$ & $\geq 2$ & $2$ & $\geq 3$ & $3$ & $4$ & $\geq 5$ & $5$ \\
    \hline
    $\delta$  & $0$ & $n$ & $2$ & $3$ & $4$ & $6$ & $n+6$ & $8$ & $9$ & $10$ \\
    \hline
    $\alpha(\Phi)$ & $0$ & $0$ & $1$ & $1$ & $2$ & $2$ & $2$  & $3$ & $3$ & $4$ \\
    \hline
    $\beta(\Phi)$  & $0$ & $0$ & $1$ & $2$ & $2$ & $3$ & $3$ & $4$ & $5$ & $5$ \\
    \hline
    $\delta(\Phi)$ & $0$ & $0$ & $2$ & $3$ & $4$ & $6$ & $6$ & $8$ & $9$ & $10$ \\
    \hline
    $\zeta(\Phi)$ & $0$ & $0$ & $1$ & $2$ & $3$ & $4$ & $4$ & $6$ & $7$ & $8$ \\
    \hline
    \end{tabular}
  \vspace{1ex}
  \caption{Basic data associated to additive fiber types}
  \label{table:1}
\end{table}

For instance, if the special fiber of the minimal Weierstrass equation above has Kodaira fiber type $\Phi=\III$, then according to Table \ref{table:1} we must have $\alpha=1$, $\beta\geq 2$, and $\delta=3$. We also introduce the quantities $\alpha(\Phi),\beta(\Phi),\delta(\Phi),$ and $\zeta(\Phi)$, defined according to the fourth through seventh rows of Table \ref{table:1}. These are respectively the minimum possible values of the quantities $\alpha$, $\beta$, and $\delta$ among all fibrations with a fiber of type $\Phi$. In particular, these depend only on the fiber type $\Phi$, and not on the specifics of a minimal Weierstrass equation having reduction type $\Phi$.  % (here we do not distinguish between types $\I_n$ or $\I_n^*$ for $n=0$ and $n\geq 1$), 
For the additive fiber types we have
\[
    \zeta(\Phi)=\alpha(\Phi)+\beta(\Phi)-1.
\]
Furthermore, we observe that if $\Phi$ is the Kodaira type of the special fiber of the fibration with the above Weierstrass equation~\eqref{eq:minimal weierstrass equation}, then we have
\begin{equation}\label{eq:delta relation}
    \min\{3\alpha, 2\beta\}=\delta(\Phi)
\end{equation}
In particular, the left side of~\eqref{eq:delta relation} depends only on the Kodaira type $\Phi$. (Note that we use $\alpha$ and $\beta$, not $\alpha(\Phi)$ and $\beta(\Phi)$, so this could \emph{a priori\/} depend upon more than just the fiber type $\Phi$.)

Given a list of Kodaira fiber types $\bPhi$ (possibly with multiplicities), we let $\alpha(\bPhi)$ denote the sum of the $\alpha(\Phi_i)$ where $\Phi_i$ ranges over the elements of $\bPhi$ (with multiplicity). We similarly define $\beta(\bPhi),\delta(\bPhi),$ and $\zeta(\bPhi)$.

%The fourth and fifth rows of Table \ref{table:1} show the values of $\alpha$ and $\beta$ needed to ensure the desired valuation of $\Delta$. An entry of ``$\geq i$'' indicates that the value in question must be at least equal to $i$, possibly including $\infty$.

\subsection{A criterion for Frobenius splitting}
We now consider a Jacobian elliptic surface $f:X\to\P^1$. Using our computations with
divisors in $\P^1$ of Section \ref{sec:frob-split}, we will derive explicit
conditions for $f$ to be Frobenius split. These conditions will be easily
computable in practice, given a Weierstrass equation for $f$. In particular,
we will see that the Frobenius splitting of $f$ is controlled by its additive
singular fibers. For technical reasons, we will focus on the question of when
$f$ is $2e$-Frobenius split. A similar analysis is possible for the odd
iterates, with a few modifications arising from the fact that fiber types can
change under pullback by odd powers of Frobenius. Given this, the analysis is
substantially the same.

% %Fix a point $x\in\P^1$, let $R$ be the local ring of $\P^1$ at $x$, let $t\in R$ be a uniformizer, and let $K$ be the field of fractions of $R$. Consider the minimal Weierstrass equation
% 	\begin{equation}\label{eq:minimal weierstrass equation first appearence}
%         y^2=x^3+gt^\alpha x+ht^\beta
%     \end{equation}
% for the elliptic curve $X_K$, where $g$ and $h$ are units in $R$ and $\alpha,\beta\in\N\cup\{\infty\}$, with the convention that $t^\infty=0$. Minimality is equivalent to the condition that either $\alpha<4$ or $\beta<6$ (or both). The discriminant $\Delta$ of $X_K$ is given by the formula 
%     \[
%         \Delta=4g^3t^{3\alpha}+27h^2t^{2\beta}
%     \]
% Let $\delta$ be the valuation of $\Delta$, and let $m=\min(3\alpha,2\beta)$.
% \begin{lem}
%     The quantity $m$ is equal to $\delta$ if the fiber of $X$ over $x$ does not have Kodaira type $\I_n^*$, and is equal to $6$ if the Kodaira type is $\I_n^*$. In particular, $m$ depends only on the Kodaira fiber type, and not on the local behavior of $X$ near $x$, and $m=0$ if the fiber is smooth.
% \end{lem}
% \begin{proof}
%     This is a consequence of Tate's algorithm, which determines the fiber type of a minimal Weierstrass equation from the data $\alpha,\beta,$ and $\delta$. The resulting restrictions are recorded in Table \ref{table:1} of Section \ref{sec:fiber moduli}.
% \end{proof}

Given a point $x\in\P^1$, we let $\Phi_x$ denote the Kodaira fiber type of the fiber of $X$ over $x$.
\begin{defn}\label{def:Delta +}
    We define a divisor $\FT(f)$ on $\P^1$ by
    \[
        \FT(f):=\sum_{x\in \P^1} \delta(\Phi_x) x.
    \]
\end{defn}
  
%Let $A\subset\P^1(k)$ be the set of points over which the fiber of $f$ is additive. For $a\in A$, let $\Phi_a$ denote the Kodaira fiber type of the fiber at $a$, and let $\delta(\Phi_a)$ denote the valuation of the discriminant of $f$ at $a$ if the fiber type is not $\I_n^\ast$ and $6$ if the fiber is of type $\I_n^\ast$. (The Kodaira fiber types $\Phi_a$ and the values of the $\delta(\Phi_a)$ are recorded in Table \ref{table:1} of Section \ref{sec:fiber moduli}.)

The significance of this divisor is in the following result.

\begin{lem}\label{lem:frobenius pullback has negative thing}
	For each $e>0$, there is a canonical isomorphism
	$$\iota_{2e}:\R^1(\phi_{2e})_\ast\ms
	O_{Y_{2e}}\simto((F^{2e})^\ast\R^1f_\ast\ms
	O_X)\left(\frac{p^{2e}-1}{12}\FT(f)\right).$$
\end{lem}
\begin{proof}
	%For the sake of notational simplicity, we will prove this for $X$ and $Y=Y_2$ (i.e., $e=2$). The general case follows by iteration.
	
% 	The basechange map gives a natural injection
% 	$$\beta:F^{2\ast}\R^1f_\ast\ms O_X\to\R^1\phi_\ast\ms O_Y,$$ from which we
% 	conclude that there is an expression
% 	$$\R^1\phi_\ast\ms O_Y\cong F^{2\ast}\R^1f_\ast\ms O_X(c)$$ for some $c$. More
% 	precisely, the basechange map is an isomorphism over the locus of semistable
% 	fibers (since the singularities are rational). To find $c$, we can work
% 	locally and describe the length of the cokernel of $\beta$ around each
% 	additive fiber. (This also produces an explicit divisor, which is stronger
% 	than simply knowing that the latter sheaf is a twist.)
	The basechange map gives a natural injection
	$$\psi:(F^{2e})^\ast\R^1f_\ast\ms O_X\to\R^1\phi_{2e,\ast}\ms O_{Y_{2e}},$$ from which we
	conclude that there is an expression
	$$\R^1\phi_{2e,\ast}\ms O_{Y_{2e}}\cong (F^{2e})^{\ast}\R^1f_\ast\ms O_X(c)$$ for some $c$. More precisely, the basechange map is an isomorphism over the locus of smooth fibers (in fact, we will see that it is an isomorphism also over the locus of semistable
	fibers, because the corresponding singularities are rational). To find $c$, we can work
	locally and describe the length of the cokernel of $\psi$ around each
	additive fiber. (This also produces an explicit divisor, which is stronger
	than simply knowing that the latter sheaf is a twist.)
	
	Let $x\in\P^1$ be a point supporting an additive fiber of $X$. Let $R$ be the local ring of $\P^1$ at $x$, let $t\in R$ be a uniformizer, and let $K$ be the field of fractions of $R$. If the elliptic curve $X_K$ is described by the Weierstrass equation~\eqref{eq:minimal weierstrass equation}, then the pulled back elliptic curve $(X_K)^{(p^{2e}/K)}$ is given by the Weierstrass equation
    \begin{equation}\label{eq:minimal weierstrass equation pulled back}
        y^2=x^3+gt^{p^{2e}\alpha}x+ht^{p^{2e}\beta}
    \end{equation}
    Let $\alpha'=p^{2e}\alpha$ and $\beta'=p^{2e}\beta$ be the exponents of the pulled back Weierstrass equation. This equation may no longer be minimal. That is, we may have $\alpha'\geq 4$ and $\beta'\geq 6$. To make a minimal equation, we repeatedly apply the change of variables
    \[
        x\mapsto t^2x\hspace{1cm}y\mapsto t^3y
    \]
    and divide the resulting equation by $t^6$, which has the effect of lowering $\alpha'$ by $4$ and $\beta'$ by $6$, until at least one of $0\leq\alpha'<4$ and $0\leq\beta'<6$ hold. The number of times we need to apply this change of variables is the unique positive integer $\lambda$ such that at least one of the inequalities
    \[
        0\leq p^{2e}\alpha-4\lambda<4\hspace{1cm}\mbox{and}\hspace{1cm}0\leq p^{2e}\beta-6\lambda<6
    \]
    holds, which is
    \[
        \lambda=\frac{p^{2e}-1}{12}\min\{3\alpha,2\beta\}
    \]
    It follows from Table \ref{table:1} that $\delta(\Phi_x)=\min\{3\alpha,2\beta\}$~\eqref{eq:delta relation}. Moreover, as explained by Schröer \cite[Theorem 10.1]{1703.03081}, this change of variables corresponds geometrically to performing a certain blow up and blow down of the corresponding surface, and each iteration increments the local value of $c$ by $1$. Adding up the local contributions we get the result.
    %we factor out powers of $t^{12}$ until
	%$p^2\delta$ is between $0$ and $11$. Each division by $t^{12}$ corresponds to
	%running the Tate algorithm once.Since $p^2\equiv 1\pmod{12}$, we have the tautological equation
	%$$\delta=p^2\delta-\left(\frac{p^2-1}{12}\delta\right)12,$$ which shows that the Tate
	%algorithm is run $\frac{p^2-1}{12}\delta$ times (around $x$). Applying this for all
	%$x$ supporting additive fibers, we see that 
	%$$c=\frac{p^2-1}{12}\sum \delta_x x=\frac{p^2-1}{12}\FT(f),$$ as claimed.
\end{proof}

The following result gives an explicit criterion for Frobenius splitness of a Jacobian elliptic surface.
\begin{prop}\label{prop:cohoinj}
	Let $f:X\to\P^1$ be a Jacobian elliptic surface.
	\begin{enumerate}
		\item $f$ is $2e$-Frobenius split if and only if the divisor
		$\frac{p^{2e}-1}{12}\FT(f)$ is $2e$-Frobenius split.
		\item $f$ is $\infty$-Frobenius split if and only if the divisor
		$\frac{p^{2e}-1}{12}\FT(f)$ is $2e$-Frobenius split for all $e>0$.
	\end{enumerate}
\end{prop}
\begin{proof}
  Consider the base change map
	$$\R^1f_\ast\ms O_X\to F^{2e}_\ast\R^1(\phi_{2e})_\ast\ms O_{Y_{2e}}$$ By the
	projection formula and Lemma \ref{lem:frobenius pullback has negative thing},
	this morphism is identified with the twist of the canonical morphism 
	$$\chi:\ms O_{\P^1}\to F^{2e}_\ast\ms
	O_{\P^1}\left(\frac{p^{2e}-1}{12}\FT(f)\right)$$ attached to the 
	divisor $\frac{p^{2e}-1}{12}\FT(f)$ by the invertible sheaf $\R^1f_\ast\ms
	O_X$. We thus see that $f$ is $2e$-Frobenius-split if and only if
	$\frac{p^{2e}-1}{12}\FT(f)$ is $2e$-Frobenius split, which proves (1). Claim
	(2) follows from (1) and Lemma \ref{lem:level lowering}.
\end{proof}

\begin{prop}\label{prop:cohoinj2}
	If $f:X\to\P^1$ is a Jacobian elliptic surface with at most two additive
	fibers, then $f$ is $\infty$-Frobenius split.
\end{prop}
\begin{proof}
	By composing with an automorphism of $\P^1$, we may assume without loss of
	generality that all additive fibers of $f$ occur over $0$ and $\infty$. The
	divisor $\FT(f)$ is then represented by a polynomial of the form
	$P(s,t)=s^{\lambda_1}t^{\lambda_2}$, and we have $\lambda_1,\lambda_2\leq 10$
	(see Table \ref{table:1}). Let $e$ be a positive integer. We have that
	\[
	  \lambda_i\frac{p^{2e}-1}{12}<p^{2e}
	\]
	which implies that $P(s,t)^{\frac{p^{2e}-1}{12}}$ is not divisible by
	$s^{p^{2e}}$ nor $t^{p^{2e}}$. By Proposition \ref{prop:frob-split},
	$\frac{p^{2e}-1}{12}\FT(f)$ is $p^{2e}$-split. By Proposition
	\ref{prop:cohoinj} (2) we conclude that $f$ is $\infty$-Frobenius split.
\end{proof}

\section{Some explicit examples}\label{sec:explicit examples}
In this section, we discuss the Frobenius splitting behavior of some explicit
elliptic fibrations. In light of our applications in the following sections we will focus on examples of fibrations on supersingular K3 surfaces. We will consider the equations for such fibrations obtained by Shioda \cite{MR918849}, which are not only supersingular but exhibit a range of Artin invariants. We continue to assume $p\geq 5$.

\begin{example}[Shioda]\label{ex:shioda example1} Consider the Weierstrass
	equation
	$$
	y^2=x^3+t^7\cdot x+t^2
	$$
	The minimal resolution of the corresponding projective surface is a K3 surface $X$. Let $f:X\to\P^1$ be the resulting elliptic fibration over the projective line $\P^1$ with homogeneous coordinates $[s:t]$. If $p\neq 17$, then this fibration has one singular fiber of Kodaira type
	$\IV$ (over $[1:0]$), one fiber of type $\III$ (over $[0:1]$), and 17 fibers of type
	$\I_1$ (over the points $[1:\zeta]$, where $\zeta$ is a $17$th root of $-27/4$). Thus, the divisor $\FT(f)$ is represented by the polynomial
	$s^3t^4$. As in Proposition \ref{prop:cohoinj2}, we see that $f$ is
	$\infty$-Frobenius split.
	
	This example was studied by Shioda \cite[Example 7]{MR918849}, who showed that if $p$ is not congruent to $1$ or $17$ modulo $34$ then $X$ is supersingular, with Artin invariant
	determined by the residue class of $p$ modulo $34$ according to the following
	list.
	$$
	\sigma_0(X)=\begin{cases}
	8&\mbox{if }p\equiv 3,5,7,11,23,27,29,31\pmod{34}\\
	4&\mbox{if }p\equiv 9,15,19,25\pmod{34}\\
	2&\mbox{if }p\equiv 13,21\pmod{34}\\
	1&\mbox{if }p\equiv 33\pmod{34}
	\end{cases}
	$$
\end{example}

\begin{example}[Shioda]\label{ex:shioda example2} %{\cite[Example 9]{MR918849}}
    Consider the Weierstrass equation
	$$
	y^2=x^3+t\cdot x+t^8,
	$$
	We let $X\to\P^1$ be the corresponding elliptic K3 surface, where as before the fibration is over the projective line $\P^1$
	with homogeneous coordinates $[s:t]$. If $p\neq 13$, then this fibration has one singular fiber of Kodaira type
	$\III$ (over $[1:0]$), one fiber of type $\IV^*$ (over $[0:1]$), and 13 of
	type $\I_1$ (over the points $[1:\zeta]$, where $\zeta$ is a $13$th root of $-4/27$). Thus, the divisor $\FT(f)$ is represented by the polynomial
	$s^8t^3$. As in Proposition \ref{prop:cohoinj2}, we see that $f$ is
	$\infty$-Frobenius split.
	
	Moreover, Shioda \cite[Example 9]{MR918849} shows that, if $p\equiv 7,11,15,19\pmod{26}$, then $X$ is a supersingular K3 surface of Artin invariant $\sigma_0=6$.
	%Suppose $p$ is a prime such that the minimal resolution is a K3 surface $X$.
\end{example}

\begin{example}[Shioda]\label{ex:shioda example3}
    Consider the Weierstrass equation
	$$
	y^2=x^3+t^{11}+1
	$$
	and as before let $X\to\P^1$ be the corresponding elliptic K3 surface. If $p\neq 11$, then this is an isotrivial elliptic fibration
	with twelve singular fibers of Kodaira type $\II$, located at $[0:1]$ and the points $[1:\zeta]$ where $\zeta$ is an $11$th root of $-1$. The divisor $\FT(f)$ is
	given by
	$$
	  P(s,t)=s^2(s^{11}+t^{11})^2=s^2t^{22}+2s^{13}t^{11}+s^{24}
	$$
	Let $e$ be a positive integer. Every term of the
	product
	\[
	  P(s,t)^{\frac{p^{2e}-1}{12}}=s^{\frac{p^{2e}-1}{6}}(s^{11}+t^{11})^{\frac{p^{2e}-1}{6}}
	\]
	is divisible by $s^{p^{2e}}$ or $t^{p^{2e}}$. By Proposition
	\ref{prop:cohoinj}, $f$ is not $2e$-Frobenius split. By Lemma \ref{lem:level
	lowering}, we conclude that $f$ is not $e$-Frobenius split for any $e\geq 2$.
	
	As explained in Example 8 of \cite{MR918849}, if $p\equiv
  17,29,35,41\pmod{66}$, then $X$ is a supersingular K3 surface of Artin
  invariant $\sigma_0=5$, and if $p\equiv 65\pmod{66}$ then $X$ is a
  supersingular K3 surface of Artin invariant $\sigma_0=1$. 
\end{example}

\section{Purely inseparable multisections on elliptic supersingular K3 surfaces}
\label{sec:main event}

In this section, we derive some properties of pullback maps on cohomology for $e$-Frobenius split elliptic supersingular K3
surfaces.

We begin by recalling the Artin-Tate isomorphism. Let $f\colon X\to C$ be a
Jacobian elliptic fibration. The Leray spectral sequence for the sheaf $\G_m$ on
the morphism $f\colon X\to C$ induces an isomorphism
\begin{equation}\label{eq:Artin-Tate1}
\Br(X)\xrightarrow{\sim}\H^1(C,\Pic_{X/C}^\circ)
\end{equation}
(see \cite[\S 3]{MR1610977}). 
Given any Brauer class $\alpha\in\Br(X)$, the Artin-Tate isomorphism
(\ref{eq:Artin-Tate1}) produces a $\Pic_{X/C}^\circ$-torsor
$X_{\alpha}^{\circ}\to C$. Among all compactifications of $X_{\alpha}^{\circ}$
to elliptic surfaces over $C$, there exists a unique minimal elliptic surface
$f_{\alpha}\colon X_{\alpha}\to C$ such that the inclusion
$X_{\alpha}^\circ\subset X_{\alpha}$ is equal to the smooth locus of
$f_{\alpha}$ (see \cite[\S 1]{MR0417182}, especially the material after
Proposition 1.4). We refer to the image of $\alpha$ under the map
\begin{equation}\label{eq:Artin-Tate2}
\alpha\mapsto \left(f_{\alpha}\colon X_{\alpha}\to C\right)
\end{equation}
as its \textit{associated minimal elliptic surface}. Note that the associated
minimal surface of the zero class is isomorphic to the relatively minimal model of
$X$.

\begin{lem}\label{lem:Brauer classes and sections}
	Let $f\colon X\to C$ be a Jacobian elliptic surface and $\alpha\in\Br(X)$ a
	Brauer class. The following are equivalent.
	\begin{enumerate}
		\item $\alpha=0$
		\item The associated elliptic surface $f_{\alpha}\colon X_{\alpha}\to C$ is
		Jacobian.
		\item The generic fiber $X_{\alpha,k(C)}\to\Spec k(C)$ has a $k(C)$-rational
		point.
	\end{enumerate}
\end{lem}
\begin{proof}
	It is immediate from the Artin-Tate isomorphism (\ref{eq:Artin-Tate1}) that
	$\alpha=0$ if and only if the torsor $X_{\alpha}^\circ\to C$ admits a
	section. If this is the case, then $f_{\alpha}\colon X_{\alpha}\to C$ is
	Jacobian. Conversely, any section of $f_{\alpha}$ must be contained in the
	smooth locus of $f_{\alpha}$, and hence gives rise to a section of the torsor,
	so $(1)$ is equivalent to $(2)$. Clearly $(2)$ implies $(3)$, and conversely
	taking the closure of a rational point shows $(3)$ implies $(2)$.
\end{proof}

\begin{prop}\label{prop:meat}
	Let $f\colon X\to C$ be a Jacobian elliptic surface and $\alpha\in\Br(X)$ a
	Brauer class. The associated minimal elliptic surface $f_\alpha\colon
	X_\alpha\to C$ admits a purely inseparable multisection if and only if there
	exists a smooth proper curve $C'$ and a purely inseparable finite cover $C'\to
	C$ such that for any resolution $Y\to X\times_CC'$ of the pullback the class
	$\alpha$ is in the kernel of the pullback map
	$$\Br(X)\to\Br(Y).$$
\end{prop}
\begin{proof}
	Note that it suffices to prove the result for any resolution, since the Brauer
	group is a birational invariant for regular schemes of dimension $2$.
	(Birational invariance of the $p$-primary part is subtle in higher dimension.)
	Given any such $C'$ and $Y$, functoriality of the Leray spectral sequence
	yields a diagram
	$$
	\begin{tikzcd}
	\Br(X)\arrow[hook]{r}\arrow{d}&\Br(X_{k(C)})\arrow{r}{\sim}\arrow{d}&\H^1(\Spec k(C),\Pic_{X_{k(C)}/k(C)}^\circ)\arrow{d}\\
	\Br(Y)\arrow[hook]{r}&\Br(Y_{k(C')})\arrow{r}{\sim}&\H^1(\Spec k(C'),\Pic_{Y_{k(C')}/k(C')}^\circ)
	\end{tikzcd}
	$$
	Here, the right vertical arrow is deduced from the isomorphism
	$Y_{k(C')}\xrightarrow{\sim} X_{k(C')}$ of generic fibers, the injectivity of
	the left horizontal arrows follows from the regularity of $X$ and $Y$, and the
	isomorphy of the right horizontal arrows follows from the Leray spectral
	sequence. It follows from Lemma \ref{lem:Brauer classes and sections} that
	$\alpha$ is in the kernel of the pullback map if and only if there exists a
	map $\Sigma$ fitting into a diagram
	$$
	\begin{tikzcd}
	&X_{\alpha}\arrow{d}{f_{\alpha}}\\
	C'\arrow{r}\arrow{ur}{\Sigma}&C
	\end{tikzcd}
	$$
	Letting $C'\to C$ range over all purely inseparable covers gives the result.
\end{proof}

\begin{lem}\label{lem:tfree}
	If $f\colon X\to C$ is a relatively minimal Jacobian elliptic surface with section
	$\sigma$ such that $\sigma(C)$ has negative self-intersection, then the
	induced map $\Pic_C^\circ\to\Pic_X^\circ$ is an isomorphism and the quotient
	$\Pic_X/\Pic_X^\circ$ is torsion free.
\end{lem}
\begin{proof}
	This is \cite[Lemma VII.1.2]{MR1078016}. Note that while Miranda seems to make
	a uniform (often unstated) assumption in \cite{MR1078016} that the base field
	is $\C$, the results we quote in this paper do not depend upon that
	assumption, as one can see from the proofs. We do freely use the assumption
	that the base field has characteristic at least $5$.
\end{proof}

\begin{lem}\label{lem:negative thing}
	If $f:X\to\P^1$ is a Jacobian elliptic fibration with at least one singular
	fiber, then, in the notation of Diagram \eqref{eq:the situation room}, for all
	$e>0$ we have that 
	$$\deg\R^1(\phi_e)_\ast\ms O_{Y_e}<0.$$
\end{lem}
\begin{proof}
	As explained in \cite[Lemma II.5.7]{MR1078016}, the degree is at most $0$, and it is $0$ if
	and only if the family has all smooth fibers. (The proof there, while it may
	be implicitly stated over a field of characteristic $0$, is true over any
	field of characteristic at least $5$.) On the other hand, we can compute the
	discriminant of $Y_e\to\P^1$ around a point $x\in\P^1$ (with local uniformizer
	$t$) as follows. If $\Delta$ is the discriminant of a minimal Weierstrass model for $X\to\P^1$, pullback by
	$F^e$ gives a discriminant of $\Delta^{p^e}$ for the naïve pullback of the
	Weierstrass equation. As in the proof of Lemma \ref{lem:frobenius pullback has negative thing}, to make the pulled back curve minimal near $x$, one
	scales the coordinates $x$ and $y$ by powers of $t$ so that the discriminant
	gets scaled by a power of $t^{12}$. Since $p$ is prime to $12$, it is impossible
	to make $\Delta^{p^e}t^{12n}$ invertible at $x$, whence $Y_e$ still has a
	singular fiber at $x$.
\end{proof}

We now consider a Jacobian elliptic supersingular K3 surface $f\colon X\to\P^1$.
Suppose $$\alpha\in\Br(X\tensor k\llbracket t\rrbracket)$$ is a Brauer class
(which will eventually be the class defined by \cite[Proposition
2.2.4]{1804.07282}, representing the universal formal Brauer class of $X$).
By the modular description of the Artin--Tate isomorphism \cite[Proposition
4.3.16]{1804.07282}, $\alpha$ comes with an associated family of elliptic K3
surfaces
$$
\mf X\to\Spec k\llbracket t\rrbracket \times\P^1
$$
over $\Spec k\llbracket t\rrbracket$ whose smooth locus is a family of torsors
under the smooth locus of the original Jacobian fibration $X\to\P^1$.
Restricting to a geometric generic fiber over a chosen algebraic closure
$k\llparen t\rrparen \subset\overline{k\llparen t\rrparen }$, we find an elliptic K3 surface
$$
\mf X_\infty=\mf X\tensor\overline{k\llparen t\rrparen }\to\P^1_{\overline{k\llparen t\rrparen }}
$$
over the field $\overline{k\llparen t\rrparen }$, which is non-Jacobian whenever $\alpha$ is
non-trivial (and thus the associated torsor is non-trivial). We wish to
understand when this elliptic surface admits a purely inseparable multisection.

Write $\alpha_\eta\in\Br(X\tensor_k k\llparen t\rrparen )$ for the generic value of the Brauer
class, and let $\alpha_{\infty}\in\Br(X\tensor\overline{k\llparen t\rrparen })$ denote its
restriction to the given geometric generic fiber.

\begin{prop}\label{prop:cheese}
	Suppose that the elliptic surface $\mf X_\infty\to\P^1_{\overline{k\llparen t\rrparen }}$
	admits a purely inseparable multisection. There exists a finite purely
	inseparable morphism $C\to\P^1$ from a smooth proper curve over $k$ and a
	resolution $Y\to X\times_{\P^1}C$ of the pullback such that the class
	$\alpha_{\eta}$ is in the kernel of the pullback map
	$$
	\Br(X\tensor_k k\llparen t\rrparen )\to\Br(Y\tensor_k k\llparen t\rrparen )
	$$
\end{prop}
\begin{proof}
	Suppose that $\mf X_{\infty}$ admits a purely inseparable multisection. By
	Proposition \ref{prop:meat}, there exists a smooth proper curve $C_{\infty}$
	and a finite purely inseparable morphism $\pi\colon
	C_{\infty}\to\P^1_{\overline{k\llparen t\rrparen }}$ such that for any resolution of the
	pullback of $X\tensor\overline{k\llparen t\rrparen }$ by $\pi$, the class $\alpha_{\infty}$
	is in the kernel of the pullback map. The curve $C_{\infty}$ has genus 0, so
	by Tsen's theorem is isomorphic to $\P^1_{\overline{k\llparen t\rrparen }}$. By
	\cite[\href{https://stacks.math.columbia.edu/tag/0CCZ}{Tag
	0CCZ}]{stacks-project}, we can find isomorphisms identifying $\pi$ with the
	$n$th relative Frobenius of $\P^1_{\overline{k\llparen t\rrparen }}$ over
	$\overline{k\llparen t\rrparen }$.
	
	Let $Y$ be a smooth resolution of the pullback of $X$ by the $n$th relative
	Frobenius of $\P^1$ over $k$. Consider the diagram
	$$
	\begin{tikzcd}
	\Br(X\tensor_k\overline{k\llparen t\rrparen })\ar[r] &
	\Br(Y\tensor_k\overline{k\llparen t\rrparen }) \\
	\Br(X\tensor_k{k\llparen t\rrparen })\ar[u]\ar[r] &
	\Br(Y\tensor_k{k\llparen t\rrparen })\ar[u]
	\end{tikzcd}
	$$
	By Lemma \ref{lem:negative thing} and Lemma \ref{lem:tfree}, we have that
	$\Pic_X$ and $\Pic_{Y}$ are both finitely generated free constant group
	schemes. Thus, $$\H^1(\Spec k\llparen t\rrparen ,\Pic_X)=0=\H^1(\Spec k\llparen t\rrparen , \Pic_{Y}),$$
	and it follows from the Leray spectral sequence that the vertical maps are
	injective. Note that $Y\tensor\overline{k\llparen t\rrparen }$ is a resolution of the
	pullback of $X\tensor\overline{k\llparen t\rrparen }$ along $\pi$ (since $Y$ and $X$ are
	defined over the algebraically closed field $k$, where regular and of finite
	type implies smooth), and therefore $\alpha_{\infty}$ is in the kernel of the
	upper horizontal map. It follows that $\alpha_{\eta}$ is in the kernel of the
	lower horizontal map, as desired.
\end{proof}

We come to our main result of this section.

\begin{thm}\label{thm:twistor}
	Let $X$ be an elliptic supersingular K3 surface over $k$. Suppose
	$$f:X\to\P^1$$ is a Jacobian elliptic fibration such that the map
	$\gamma_e(f)$ in \eqref{eq:cohomap} is injective for all $e\geq 1$. Suppose
	$$\alpha\in\Br(X\tensor k\llbracket t\rrbracket )$$ is a Brauer class such
	that
	$$\alpha_{X\tensor k[t]/(t^2)}\neq 0$$ but 
	$$\alpha_X=0.$$ If $\mf X\to\P^1_{k\llbracket t\rrbracket }$ is the
	Artin--Tate family of elliptic surfaces associated to $\alpha$ then the
	geometric generic fiber $\mf X_\infty\to\P^1_{\overline{k\llparen t\rrparen }}$ has no
	purely inseparable multisections.
\end{thm}
\begin{proof}
	Suppose $C\to\P^1$ is a purely inseparable cover. By
	\cite[\href{https://stacks.math.columbia.edu/tag/0CCZ}{Tag
	0CCZ}]{stacks-project}, we may assume that $C=\P^1$ and the map is the $e$th
	power of the Frobenius map. Let $Y=Y_e\to\P^1$ be the resolution studied in
	Section \ref{sec:frobenius split elliptic fibrations}. Consider the 
	kernel of the pullback map
	$$
	\Br(X\tensor k\llbracket t\rrbracket)\to\Br(Y\tensor k\llbracket t\rrbracket).
	$$
	Restricting to $k[\eps]:=k[t]/(t^2)$, we get a diagram
	$$
	\begin{tikzcd}
	0\ar[r] &\H^2(X,\ms O)\ar[r]\ar[d]&\Br(X\tensor k[\eps])\ar[r]\ar[d] &\Br(X)\ar[r]\ar[d] & 0\\
	0\ar[r] &\H^2(Y,\ms O)\ar[r]&\Br(Y\tensor k[\eps])\ar[r] &\Br(Y)\ar[r] & 0
	\end{tikzcd}
	$$
	in which the vertical arrows are the usual pullback maps. It follows from the
	hypotheses on $\alpha$ that the restriction
	$$\alpha_{k[\eps]}\in\Br(X\tensor k[\eps])$$ of $\alpha$ is the image of a
	generator for $\H^2(X,\ms O)$. Since each $\gamma_e(f)$ is injective, the left
	vertical map is injective, and we conclude that $\alpha_{Y\tensor k\llbracket
	t\rrbracket}$ is non-zero. Since $Y$ is smooth over $k$, we have that
	$Y\tensor k\llbracket t\rrbracket$ is regular, and thus 
	$$\Br(Y\tensor k\llbracket t\rrbracket)\to\Br(Y\tensor k\llparen t\rrparen )$$ is injective.

	Putting everything together, we conclude that $\alpha|_{Y_e\tensor k\llparen t\rrparen }\neq
	0$ for all $e$. By Proposition \ref{prop:cheese}, we see that the original
	fibration $\mf X_\infty\to\P^1_{\overline{k\llparen t\rrparen }}$ cannot have any purely
	inseparable multisections, as claimed.
\end{proof}

\begin{cor}\label{cor:twistor}
	Suppose $f:X\to\P^1$ is a Jacobian elliptic fibration on a supersingular K3
	surface such that for all $e>0$, the map $\gamma_e(f)$ of \eqref{eq:cohomap}
	is injective. Suppose $\alpha\in\Br(X\tensor k\llbracket t\rrbracket)$ is an
	algebraization of the universal formal Brauer class. If $\mf
	X\to\P^1_{k\llbracket t\rrbracket}$ is the Artin--Tate family associated to
	$\alpha$, then the geometric generic fibration $$\mf
	X_\infty\to\P^1_{\overline{k\llparen t\rrparen }}$$ does not admit any purely inseparable
	multisections.
\end{cor}
\begin{proof}
	This follows immediately from Theorem \ref{thm:twistor}, using the fact that
	$$\alpha_{X\tensor k[t]/(t^2)}\in\widehat{\Br}(k[t]/(t^2))=\H^2(X,\ms O_X)$$ is
	a generator.
\end{proof}
Note that the surfaces of Example \ref{ex:shioda example1} and Example
\ref{ex:shioda example2} satisfy the assumptions of Corollary \ref{cor:twistor}.
Hence, we have shown in particular that there exists a non-Jacobian elliptic
supersingular K3 surface over the algebraically closed field $\overline{k\llparen t\rrparen}$
with no purely inseparable multisections. We next show that there exist such
surfaces over $k$.

\begin{lem}\label{lem:hom scheme thing}
	Let $X\subset\P^n$ be a projective scheme of finite presentation over a scheme
	$U$ that admits a map $X\to\P^1$. Fix a morphism $f:\P^1\to\P^1$. For a fixed
	integer $e$, let $V_e^f\subset\Hom(\P^1_U,X)$ be the closed subscheme
	parametrizing morphisms $\P^1\to X$ such that the composition $\P^1\to
	X\to\P^1$ equals $f$ and the image $\P^1\to X\to\P^n$ has degree $e$. The
	scheme $V_e^f$ is of finite type over $U$.
\end{lem}
\begin{proof}
	Since $\P^1$ is separated, the condition that the composition $\P^1\to
	X\to\P^1$ equal $f$ is closed. Thus it suffices to show that the scheme
	$\Hom_e(\P^1_U,X)$ parametrizing maps of degree $e$ is of finite type. Since
	$\Hom(\P^1_U,X)\subset\Hom(\P^1_U,\P^n_U)$ is a closed immersion, it suffices
	to show the same thing for
	$$\Hom_e(\P^1_U,\P^n_U)=\Hom_e(\P^1,\P^n)\times U,$$ whence it suffices to
	show it for the absolute scheme $\Hom_e(\P^1,\P^n)$. By the universal property
	of projective space, this is given by the scheme of quotients $\ms
	O^{n+1}\to\ms O(e)$ on $\P^1$. This $\Quot$-scheme is of finite type by
	\cite[\href{https://stacks.math.columbia.edu/tag/0DPA}{Tag
	0DPA}]{stacks-project}.
\end{proof}

\begin{cor}\label{cor:they exist!}
	If $k$ is uncountable, then there exists a supersingular K3 surface $X$ over $k$ and a non-Jacobian
	elliptic fibration $f:X\to\P^1$ with no purely inseparable multisections.
\end{cor}
\begin{proof}
	Suppose that $f:X\to\P^1$ is an elliptic supersingular K3 surface such that
	$f$ is $\infty$-Frobenius split (see for instance Example \ref{ex:shioda
	example1} and Example \ref{ex:shioda example2}). As shown in Section 4.3 of
	\cite{1804.07282}, there exists a class $\alpha'\in\H^2(X\times\A^1,\m_p)$
	which restricts to a class over $X\otimes k[[t]]$ whose associated Brauer
	class is the universal Brauer class. Moreover, there is a corresponding
	elliptic surface $\widetilde{\mf X}\to \P^1_{\A^1}$ over $\A^1$.
	
	Let $\ms O(1)$ be a relative polarization of
	$\widetilde{\mf X}$ over $\A^1$. Given a positive integer $e$, let
	$H_{e,b}\subset\Hom_{\A^1}(\P^1,\widetilde{\mf X})$ be the Hom scheme
	parametrizing morphisms whose composition $\P^1\to\widetilde{\mf X}\to\P^1$ is
	the $e$th relative Frobenius and whose image has $\ms O(1)$-degree at most $b$
	in each fiber. By Lemma \ref{lem:hom scheme thing}, $H_{e,b}\to \A^1$ is of
	finite type. By Corollary \ref{cor:twistor}, the image of $H_{e,b}$ does not
	contain the generic point of $\A^1$, and hence the image is finite. Applying
	this for all $e$ and $b$, we see that the fiber of $\widetilde{\mf X}$ over a very
	general point $k$-point of $\A^1$ is a non-Jacobian elliptic supersingular K3
	surface that does not admit a purely inseparable multisection.
\end{proof}

We will show in Section \ref{sec:very general} that in fact a very general
supersingular K3 surface has the property that every non-Jacobian elliptic
fibration has no purely inseparable multisections.

As a positive result, we note the following.
\begin{prop}\label{prop:isotrivial pi}
	If $f:X\to\P^1$ is an isotrivial Jacobian elliptic fibration on a supersingular K3 surface with supersingular generic fiber, then
	\begin{enumerate}
		\item any form of $f$ admits a purely inseparable multisection of degree
		$p^2$, and
		\item $f$ is not $e$-Frobenius split for any $e\geq 2$.
	\end{enumerate} 
\end{prop}
\begin{proof}
	Let $g:X'\to\P^1$ be an elliptic K3 surface that is a form of $f$. The Brauer
	group of a supersingular K3 surface is $p$-torsion \cite[Theorem 4.3]{Artin74}, so $g$ admits a section
	after pullback along $X[p]\to\P^1$, where $X[p]$ is the $p$-torsion group
	scheme of the smooth locus of $f$. The generic fiber of $f$ is supersingular,
	so $X[p]\to\P^1$ is purely inseparable of degree $p^2$. It follows that the
	map $Y_{2}\to X$ induces the zero map on formal Brauer groups, and therefore
	the pullback $\H^2(X,\ms O_X)\to\H^2(Y_2,\ms O_{Y_2})$ vanishes. By Lemma
	\ref{lem:frob split}, the fibration $f$ is not $e$-Frobenius split for any
	$e\geq 2$.
\end{proof}
\begin{remark}
	In fact, we do not know of a single example of a non-isotrivial fibration on a
	supersingular K3 surface that is not Jacobian and admits a purely inseparable
	multisection.	
\end{remark}

\section{Moduli of Weierstrass data}
\label{sec:moduli}

In this section we recall some of the basic theory of Weierstrass fibrations and
their associated data. This is inspired by the discussion in \cite{MR615858}
(see also \cite{MR0417182}) and mainly serves to fix notation. Fix an
algebraically closed field $k$ of characteristic not equal to $2$ or $3$. For
technical reasons, we briefly describe how to give a relative construction over
$k$; this is purely linguistic.

\begin{defn}
  Fix a $k$-scheme $T$. A \emph{Weierstrass fibration parametrized by $T$\/} is a
  pair $(f:X\to Y,\sigma)$ with 
  \begin{enumerate}
    \item $Y\to T$ a flat morphism of schemes with integral geometric fibers;
    \item $f:X\to Y$ a proper flat morphism of schemes of finite presentation;
    \item $\sigma:Y\to X$ a section with image in the smooth locus of $f$, such
    that 
    \item each (pointed) geometric fiber of $f$ is an elliptic curve, a nodal
    rational curve, or a cuspidal rational curve. 
  \end{enumerate}

	If $Y$ is fixed, we will called $f$ a \emph{Weierstrass fibration over $Y$\/}.
	If $Y=Z\times T$ for a fixed $k$-scheme $Z$, we will call the pair a
	\emph{Weierstrass fibration over $Z$ parametrized by $T$\/}. (In particular, a
	Weierstrass fibration over $Z$ parametrized by $T$ is the same thing as a
	Weierstrass fibration over $Z\times T$.)
  
  An \emph{isomorphism\/} of Weierstrass fibrations is a commutative diagram of
  $T$-schemes 
  \begin{equation}\label{eq:w iso}
    \begin{tikzcd}
      X\ar[r,"\alpha"]\ar[d,"f"] & X'\ar[d, "f'"'] \\
      Y\ar[r,"\beta"]\ar[u,bend left, "\sigma"] & Y'\ar[u, bend right, "\sigma'"']
    \end{tikzcd}
  \end{equation}
  in which $\alpha$ and $\beta$ are isomorphisms. If $Y=Y'$ and $f$ and $f'$ are
  considered as Weierstrass fibrations over $Y$, then we require that
  $\beta=\id$ in \eqref{eq:w iso}.
\end{defn}

Given a Weierstrass fibration $(f:X\to Y,\sigma)$, we will write $S\subset X$
for the image of $\sigma$ and call this the \emph{zero section\/}.

\begin{notation}
  We will write $\WeierstrassFibrations_Y$ for the category of Weierstrass
  fibrations over $Y$. Given a fixed $k$-scheme $Z$, we will write
  $\WeierstrassFibrations(Z)$ for the fppf $k$-stack whose fiber over a scheme
  $T$ is the groupoid $\WeierstrassFibrations_{Z\times_k T}$. There is an open
  substack $\WeierstrassFibrations(Y)^\circ$ whose objects over $T$ are
  Weierstrass fibrations over $Y$ parametrized by $T$ such that for each
  geometric point $t\to T$, the induced fibration $X_t\to Y_t$ has a smooth
  fiber.
\end{notation}

\begin{prop}\label{prop:WF is an alg stack}
  If $Z$ is a proper $k$-scheme then the stack $\WeierstrassFibrations(Z)$ is a
  separated Artin stack with finite diagonal locally of finite type over $k$.
\end{prop}
\begin{proof}
	It follows from the various definitions that there is a canonical isomorphism
  $\WeierstrassFibrations(Z)=\Hom(Z,\WeierstrassFibrations(\Spec k))$. Applying
  \cite[Theorem 1.1]{MR2239345}, we see that it suffices to establish the result
  for $Z=\Spec k$. We will write $\ms M:=\WeierstrassFibrations(\Spec k)$ for
  the sake of notational simplicity.

  Given a family $\pi:X\to T, \sigma:T\to X$ in $\ms M$, cohomology and base
  change tells us that $\pi_\ast\ms O(3\sigma(T))$ is a locally free sheaf of
  rank $3$, and that the natural morphism $X\to\Proj(\Sym^\ast\pi_\ast\ms
  O(3\sigma(T)))$ is a closed immersion. Let $G\subset\GL_3$ be the subgroup
  whose image in $\PGL_2$ is the stabilizer of the point $(0:1:0)$. We can
  construct a natural $G$-torsor $M\to\ms M$ whose fiber over $(\pi,\sigma)$ is
  the scheme of isomorphisms $\pi_\ast\ms O(3S)\to\ms O^{\oplus 3}$ with the
  property that the composition $$\ms O\to\pi_\ast\ms O(3S)\to\ms O^{\oplus 3}$$
  lands in the span of $(0,1,0)$, where the first map above is the adjoint of
  the map $\ms O_X\to\ms O_X(3S)$ corresponding to the divisor $3S$. The
  $G$-torsor $M$ admits an open immersion into the space of cubic curves in
  $\P^2$ passing through $(0:1:0)$. 
  
  Consider the diagonal of $\ms M$. By well-known results (for example,
  \cite[Paragraph 4.c]{MR1611822}), we know that $\Delta:\ms M\to\ms M\times\ms
  M$ is representable by schemes of finite presentation. We claim that $\Delta$ is
  finite. To see this, it suffices to work with two families $\pi:X\to
  T,\sigma:T\to X$ and $\pi':X'\to T, \sigma':T\to X$ with $T$ the spectrum of a
  dvr. The divisors $\ms O(3S)$ and $\ms O(3S')$ define embeddings
  $X,X'\inj\P^2_T$. An isomorphism of the generic fibers over $T$ gives a
  pointed isomorphism, which gives a change of coordinates that conjugates the
  Weierstrass equations of $X_K$ and $X'_K$. Since the minimal Weierstrass
  models over $T$ are unique, this shows that the isomorphism extends uniquely
  to an isomorphism $X\to X'$, giving properness of $\Delta$. On the other hand,
  we know that the automorphism group of any fiber has size at most $6$, when we
  conclude that $\Delta$ is finite.

  It follows that we can identify $\ms M$ with the quotient stack $[M/G]$, which
  has finite diagonal. This shows the desired result. (One could also prove this
  result using Artin's Representability Theorem, and checking various
  deformation-theoretic conditions, but this proof given here is more concrete
  and informative.)
\end{proof}

Now we define what we need in order to globalize the classical Weierstrass
equations of elliptic curves and their models.

\begin{defn}
  Fix a $k$-scheme $T$. A \emph{Weierstrass datum parametrized by $T$\/} is a
  triple $(L, a, b)$ with $L$ an invertible sheaf on $T$,
  $a\in\Gamma(T,L^{\tensor -4})$, and $b\in\Gamma(T, L^{\tensor -6})$. The
  section 
  $$\Delta:=4a^{\tensor 3}+27b^{\tensor 2}\in\Gamma(T, L^{\tensor -12})$$ is
  called the \emph{discriminant\/} of $(L,a,b)$. An \emph{isomorphism\/}
  $(L,a,b)\to (L',a',b')$ of Weierstrass data is given by an isomorphism
  $\phi:L\to L'$ of invertible sheaves such that $\phi^{\tensor -4}(a)=a'$ and
  $\phi^{\tensor -6}(b)=b'$. 
\end{defn}

\begin{notation}
  We will write $\WeierstrassData_T$ for the groupoid of Weierstrass data
  parametrized by $T$. Given an integral $k$-scheme $Y$ of finite type, we will
  write $\WeierstrassData(Y)$ for the fppf stack whose fiber over a $k$-scheme
  $T$ is the subgroupoid $\WeierstrassData_{Y\times_k
  T}^\circ\subset\WeierstrassData_{Y\times_k T}$ consisting of those Weierstrass
  data such that the discriminant $\Delta$ is not identically zero in any
  geometric fiber over $T$.
\end{notation}

\begin{prop}
  If $Y$ is proper over $k$ then $\WeierstrassData(Y)$ is an Artin stack locally
  of finite type over $k$.
\end{prop}
\begin{proof}
  There is a representable forgetful morphism
  $\WeierstrassData(Y)\to\Pic_{Y/\spec k}$ whose fiber over an invertible sheaf
  $L$ on $Y\times T$ is the space of pairs $(a,b)$ of sections of $L^{\tensor
  -4}$ and $L^{\tensor -6}$ such that the discriminant does not identically
  vanish, which is itself open in the space of pairs. The scheme of sections of
  an invertible sheaf is representable by a geometric line bundle. Since
  $\Pic_{Y/\Spec k}$ is an Artin stack locally of finite type over $k$, we
  conclude that the same is true for $\WeierstrassData(Y)$, as desired.
\end{proof}

Given a Weierstrass datum $(L,a,b)$ parametrized by $Y\times T$, one can associate a
Weierstrass fibration $X\to Y\times T$ parametrized by $T$ as follows. Let
$f:X\to Y\times T$ be the projection morphism associated to the relative divisor
$$y^2z=x^3+axz^2+bz^3$$ in $\P(L^{\tensor -2}\oplus L^{\tensor -3}\oplus\ms
O_{Y\times T})$. The section $\sigma$ is defined by $(x:y:z)=(0:1:0)$ (that is,
the projection on the second factor).

\begin{prop}\label{prop:weierstuff}
  Suppose $Y$ is a proper integral $k$-scheme. The morphism 
  $$\WeierstrassData(Y)\to\WeierstrassFibrations(Y)$$ described above is an
  isomorphism of stacks.
\end{prop}
\begin{proof}
  To prove the statement it suffices to prove that the functor
  $$\WeierstrassData(Y)_T\to\WeierstrassFibrations(Y)_T$$ is an equivalence
  (respecting the indicated subcategories), where $T$ is a strictly local ring.
  This then reduces to \cite[Theorem 2.1]{MR615858}, whose proof invokes
  \cite[Theorem 1${}^\prime$]{MR0437531}, whose proof is, in turn, an unrecorded
  relativization of \cite[Theorem 1]{MR0437531}). Furthermore, one can see that
  the proof of \cite[Theorem 1]{MR0437531} is a consequence of cohomology and
  base change for the fibers of Weierstrass fibrations, so it applies more
  generally over non-reduced bases (which is not \emph{a priori\/} allowed by
  the hypotheses of \cite[Theorem 1]{MR0437531} and \cite[Theorem
  2.1]{MR615858}).
\end{proof}

\begin{defn}\label{defn:minimal}
  A Weierstrass datum $(L,a,b)$ parametrized by a smooth curve $Y$ is
  \emph{minimal\/} if for all points $y\in Y$ we have that $\ord_y a<4$ or $\ord_y b < 6$.
\end{defn}

\begin{lem}\label{lem:minimal is open}
	Given an object $(L,a,b)$ of $\WeierstrassData(Y)_T$, there is an open
	subscheme $U\subset T$ such that a $T$-scheme $Z\to T$ factors through $U$ if
	and only if for all geometric points $z\to Z$, the restriction $(L,a,b)_z$ is
	minimal in the sense of Definition \ref{defn:minimal}.
\end{lem}
In other words, the minimal locus is open.
\begin{proof}
	This is a topological condition, so it suffices to show that the locus $R$ of
	points $t\in T$ such that $(L,a,b)_t$ is minimal is open. First, note that any
	Weierstrass datum is locally induced by a Weierstrass datum over a Noetherian
	scheme. Thus, we may assume $T$ is Noetherian. The locus $R$ is closed under
	generization, since the order of vanishing of a section of an invertible sheaf
	can only increase under specialization. To show the desired openness it
	suffices to show that $R$ is constructible. To do this, we may assume $T$ is
	integral. There is an open subscheme $U\subset T$ over which the vanishing
	loci of $a$ and $b$ are flat. Shrinking if necessary and making a finite flat
	extension of $U$, we may assume that in fact we can write the vanishing locus
	of $a$ as $\sum_{i\in I} a_i s_i$ and the vanishing locus of $b$ as
	$\sum_{j\in J} b_j t_j$, where $s_i$ and $t_j$ are sections of the projection
	$Y_U\to U$ and $a_i, b_j\in\N$. The points in $R_U$ parametrizing minimal
	Weierstrass data can be described as follows: given $i\in I$ and $j\in J$, let
	$Z_{ij}\subset U$ be the image of $s_i\cap t_j$, which is a closed subset. If
	$Z_{ij}\neq\emptyset$, then $R$ contains $Z_{ij}$ if and only if $a_i<4$ or
	$b_j<6$. We can thus write $R_U$ as the complement of finitely many $Z_{ij}$,
	depending upon the values of $a_i$ and $b_j$. This defines an open subset of
	$U$. By Noetherian induction, we see that $R$ is constructible, as desired.
\end{proof}

In this paper, we will focus on the case $Y=\P^1$, and we will work with
fibrations that are generically smooth. 

\begin{notation}
  We will write $\WeierstrassFibrations$ for
  $\WeierstrassFibrations(\P^1)^\circ$ and $\WeierstrassData$ for
  $\WeierstrassData(\P^1)$ in what follows. We will write $\WeierstrassData^n$
  for the substack consisting of minimal Weierstrass data $(L,a,b)$ where $L$
  has degree $n$ (i.e., $L\cong\ms O_{\P^1}(n)$).
\end{notation}

\subsection{A concrete cover of $\WeierstrassData^n$}
\label{sec:concrete}

We can describe
$\WeierstrassData^n$ concretely as follows. Let $V_i$ be the affine space whose
underlying $k$-vector space is $\Gamma(\P^1,\ms O(i))$, and write
$V_i^\ast=V^i\setminus\{0\}$ and $V_i^\aff\subset V_i^\ast$ for the set of
sections that do not vanish at $\infty\in\P^1$. There is an open subset
$U_n\subset V_{-4n}^\ast\times V_{-6n}^\ast$ consisting of pairs $a\in
V_{-4n}^\ast, b\in V_{-6n}^\ast$ such that $4a^3+27b^2\neq 0$. 

\begin{lem}\label{lem:cov}
  The universal pair $(a,b)$ over $U_n$ gives a smooth cover
  $\chi:U_n\to\WeierstrassData^n$. In addition, the fibers of $\chi$ have
  dimension $1$.
\end{lem} 
\begin{proof}
  By Proposition \ref{prop:weierstuff}, any deformation of a family is induced
  by a deformation of the sections $a$ and $b$, so $\chi$ is smooth. On the
  other hand, the fibers of $\chi$ are precisely the $\G_m$-orbits in $U$, so
  $\chi$ has relative dimension $1$.
\end{proof}

\begin{cor}\label{lem:wd-dim}
  The stack $\WeierstrassData^{-2}$ is a tame DM stack of dimension $21$.
\end{cor}
\begin{proof}
  The automorphism group of a datum $(\ms O(-2),a,b)$ is given by scalar
  multiplications $s:\ms O(-2)\to\ms O(-2)$ such that the action of $s$ on
  preserves $a$ and $b$. In particular, if $a$ and $b$ are both nonzero, then
  $s^4=1=s^{6}$. Thus, $s^2=1$, so the automorphism group is $\m_2$. Similarly,
  if $b=0$ then the automorphism group is $\m_4$ and if $a=0$ then the
  automorphism group is $\m_6$. Since $p\geq 5$, we see that
  $\WeierstrassData^{-2}$ is a tame DM stack.

  To compute the dimension, note that $V_8\times V_{12}$ has dimension $22$. By
  Lemma \ref{lem:cov}, we see that $\WeierstrassData^{-2}$ has dimension $21$,
  as desired. 
\end{proof}

\subsection{The cover by universal polynomials}
\label{sec:univ poly}

For the sake of computation, there is a further flat covering that will be
useful.

\begin{defn}\label{defn:poly}
  Let $\Poly\subset\A^{-10n+2}$ be the space of tuples
  $$(a_0,a_1,\ldots,a_{-4n},b_0,b_1,\ldots,b_{-6n})$$ such that the polynomial
  $$\Delta=4a_0^3\left(\prod_{i=1}^{-4n}(t-a_i)\right)^3-27b_0^2\left(\prod_{j=1}^{-6n}(t-b_j)\right)^2$$
  is not identically $0$ and the associated Weierstrass fibration is relatively
  minimal. (This is open by Lemma \ref{lem:minimal is open}.)
\end{defn}

There is a universal Weierstrass equation defined over $\Poly$, namely 
\begin{equation}\label{eq:weier}
  y^2=x^3+a_0\left(\prod_{i=1}^{-4n}(t-a_i)\right)x+b_0\prod_{j=1}^{-6n}(t-b_j).
\end{equation}
It has the property that the fiber over $t=\infty$ is not an additive fiber.
There is a natural decomposition
$$\Poly=\Poly^{ab\neq 0}\sqcup\Poly^{a=0}\sqcup\Poly^{b=0}.$$ The open locus
$\Poly^{ab\neq 0}$ corresponds to fibrations where $a\neq 0$ and $b\neq 0$. The
closed subsets $\Poly^{a=0}$ and $\Poly^{b=0}$ are given by the vanishing of
$a_0$ and $b_0$ (and hence the corresponding coefficient in the Weierstrass
equation), respectively. One sees that $\Poly^{a=0}$ corresponds to Weierstrass
fibrations that are isotrivial with fiber of $j$-invariant $0$, while
$\Poly^{b=0}$ parametrizes Weierstrass fibrations that are isotrivial with fiber
of $j$-invariant $1728$. Since these properties are invariant under the action
of $\PGL_2$, we see that this stratification covers a similar decomposition
$$\WeierstrassData^n=\WeierstrassData^{n,{ab\neq
0}}\sqcup\WeierstrassData^{n,a=0}\sqcup\WeierstrassData^{n,b=0}.$$

We describe a useful covering of the various strata $\WeierstrassData^{n,\star}$
by copies of the corresponding stratum $\Poly^{\star}$.

\begin{lem}\label{lem:coverings}
	Fix a value of $\star\in\{ab\neq 0, a=0, b=0\}$. There is a set of flat
  morphisms of finite presentation
	$$\rho_s^\star:\Poly^\star\to\WeierstrassData^{n,\star}, s\in S_\star$$ such
	that
  \begin{enumerate}
    \item the $\rho_s^\star$ form an fppf covering of
    $\WeierstrassData^{n,\star}$;
    \item for each $\rho_s^\star$, the pullback of the universal Weierstrass
    equation is isomorphic to \eqref{eq:weier} (with the appropriate coefficient
    equal to $0$, depending upon the value of $\star$).
  \end{enumerate}
\end{lem}
\begin{proof}
	We claim that the morphism 
	$$(a_0,a_1,\ldots,a_{-4n},b_0,b_1,\ldots,b_{-6n})\mapsto
	\left(a_0\prod_{i=1}^{-4n} (x-a_iy),b_0\prod_{j=1}^{-6n}(x-b_jy)\right)$$
	defines flat morphisms of finite presentation 
	$$\Poly^\star\to\WeierstrassData^{n,\star}$$ whose image consists of those
	Weierstrass data corresponding to fibrations that do not have an additive fiber
	over $\infty$. Assuming this, we see that composing with the elements
	$s\in\PGL_2(k)$ gives a set of morphisms $\rho_s^\star$ covering all of
	$\WeierstrassData^{n,\star}$, since any fibration possesses a smooth (in particular, non-additive)
	$k$-fiber.

	To show flatness, we first note that we can write this as a product of two
	maps (one involving the $a_i$ and one involving the $b_j$). Since the cases
	$a=0$ or $b=0$ add a trivial affine factor, it suffices to show that, in
	general, the morphism
	$$(c_0,c_1,\ldots,c_m)\mapsto c_0\prod_{\ell=1}^m(x-c_{\ell} y)$$ from the
	open subspace $c_0\neq 0$ in $\A^{m+1}$ to the space of forms $V_m$ is flat.
	The map is equivariant for the free scaling action of $\G_m$ on $c_0$, so it
	suffices to show that the induced map
	$$(c_1,\ldots,c_m)\to\prod(x-c_\ell y)$$ is flat. The coordinates on $V_m$ are
	given by the coefficients, which are the elementary symmetric functions in the
	$c_i$. This is then a morphism $\A^m\to\A^m$ with finite fibers (since
	$k[c_1,\ldots,c_m]$ is a unique factorization domain), whence it must be flat
	(since both domain and codomain are regular).
\end{proof}

This decomposition will allow us to study more easily the locus of fibrations
with a given additive fiber configuration.

\subsection{Weierstrass data attached to a family of marked surfaces}\label{sec:gen weierstrass data}

Let $X\to M$ be a family of K3 surfaces equipped with a pair of divisor classes
$a,b\in\Pic(X)$ such that for all geometric points $m\to M$, we have
\begin{enumerate}
  \item $a|_{X_m}$ is the class associated to $\ms O_{X_m}(E)$ for a smooth
  irreducible curve $E\subset X_m$ of genus $1$, and
  \item $a|_{X_m}\cdot b|_{X_m}=1$.
\end{enumerate}
Let $M'\to M$ be the stack whose fiber over $T\to M$ is the groupoid of
Weierstrass fibrations
$$(f:X\times T\to\P^1\times T,\sigma)$$ such that $[f^\ast\ms O(1)]=a$ and $\ms
O_{X\times T}(\operatorname{im}\sigma)=b$ as sections of $\Pic_{X/M}$.

\begin{lem}\label{lem:weierstrassify}
  The morphism $M'\to M$ is smooth with fibers of dimension $3$.
\end{lem}
\begin{proof}
  Let $M''\to M$ be the stack whose objects over $T\to M$ are elliptic
  fibrations $X_T\to\P^1_T$ with fibers of class $a$, without a chosen section
  $s$. 

  The morphism $M'\to M$ factors as
  $$M'\to M''\to M$$ by sending $(f:X\to\P^1,\sigma)$ to the fibration
  $(f:X\to\P^1)$ and then to the invertible sheaf $f^\ast\ms O(1)$. We claim
  that $M'\to M''$ is étale and $M''\to M$ is a $\PGL_2$-torsor.

  To see the first assertion, note that $b$ is assumed to exist everywhere over
  $M$, whence the class of the section $s$ is defined everywhere. Since $X\to M$
  is a fibration of K3 surfaces, any section in a fiber with class $b$ will
  locally deform uniquely. It follows that $M'\to M''$ is étale.

  It remains to prove that $M''\to M$ is a $\PGL_2$-torsor. Given a fibration
  $f:X\to\P^1$ with fibers of class $a$, there is a natural action of
  $\alpha\in\PGL_2$ that sends $f$ to $\alpha f$. This is an action on $M''$
  over $M$. To establish that $M=M''/\PGL_2$, it thus suffices to work
  locally and assume that $T$ is a strictly Henselian local scheme. In this
  case, every fibration $f:X_T\to\P^1_T$ with $[f^\ast\ms O(1)]=a$ is given by
  choosing a basis for $\Gamma(X_T,L)$, up to scaling. The bases are permuted
  simply transitively by $\GL_2(T)$. Dividing by scalars gives the desired
  result.
\end{proof}

	Suppose $S_n$ is the Ogus moduli space of marked supersingular K3 surfaces of
	fixed Artin invariant $n$; a point of $S_n$ is a pair
	$(X,\tau:\Lambda_n\simto\NS(X))$ consisting of a supersingular K3 surface and
	an isomorphism of the standard lattice of Artin invariant $n$ with the
	Néron--Severi lattice of $X$. Let $N\to S_n$ be the space of triples
	$((X,\tau), a, b)$ where $(X,\tau)$ is in $S_n$ and $a,b\in\Lambda_n$
	represent a Jacobian elliptic fibration -- $a$ is the class of a smooth genus
	$1$ curve, $b$ is the class of a smooth rational curve, and $a\cdot b=1$.
	There is a countable open cover $V_i\subset N$ such that $V_i$ is of finite
	type over $k$ and $V_i\to S_n$ is étale. There is a $\PGL_2$-torsor $M'_i\to
	V_i$ parametrizing Weierstrass fibrations whose fiber over $((X,\tau),a,b)$ is
	the space of fibrations $f:X\to\P^1,\sigma$ with fiber of class $a$ and
	section of class $b$. Via the equivalence of Proposition
	\ref{prop:weierstuff}, there is a resulting morphism
	$w:M_i'\to\WeierstrassData^{-2}$. Moreover, every Weierstrass datum
	corresponding to a Jacobian elliptic fibration on a supersingular K3 surface
	of Artin invariant $n$ is in the image of one of the $M_i'$.

\begin{cor}\label{cor:artin invariant weierstrassify}
	In the notation of the preceding paragraph, the morphism
	$w:M_i'\to\WeierstrassData^{-2}$ is quasi-finite.
\end{cor}
\begin{proof}
	It suffices to show that the geometric fibers of $w$ are countable, since
	$M_i'$ and $\WeierstrassData^{2}$ are of finite type over $k$. Suppose 
	$$w((X,\tau),a,b,f,\sigma)=w((X',\tau'),a',b',f',\sigma').$$ This implies that
	there is an isomorphism $g:X\to X'$ such that $f'\circ g=f$ and
	$$g^\ast(\tau'(b'))=\tau(b).$$ Since $(-2)$-curves are rigid, we then have
	that $\sigma\circ g=\sigma'$. In addition, for tautological reasons, we have
	$$g^\ast(\tau'(a'))=\tau(a).$$ (In fact, writing both $a$ and $f$ in the
	notation is redundant, but we are trying to respect $M_i'$'s origin in a tower
	of morphisms.)

	We thus find that $((X',\tau'), a', b', f', \sigma')$ is isomorphic to
	$((X,\tau''), a, b, f, \sigma)$ for some other marking
	$\tau'':\Lambda_n\simto\NS(X)$ such that $\tau''(a)=\tau(a)$ and
	$\tau''(b)=\tau(b)$. Conversely, any marking that preserves the classes of $a$
	and $b$ gives rise to a point with the same image in $\WeierstrassData^{-2}$.
	Thus, we find that the (geometric) fiber of $w$ over $w((X,\tau), a, b, f,
	\sigma)$ is identified with the subset of the automorphism group
	$\Aut(\Lambda_n)$ that fixes $a$ and $b$ and such that the resulting point
	$((X,\tau''),a,b,f,\sigma)$ lies in $M_i'$. Since $\Aut(\Lambda_n)$ is
	countable, we conclude that the fibers of $w$ are countable, as desired.
\end{proof}

\begin{cor}\label{cor:dim cram}
  The locus of Artin invariant $s$ surfaces in $\WeierstrassData^{-2}$ is a
  countable union of constructible subsets, each of which contains a locally
  closed subspace of dimension $s+2$ (equivalently, codimension $19-s$).
\end{cor}
\begin{proof}
  By Corollary \ref{cor:artin invariant weierstrassify}, the space of such
	surfaces is a countable union of images of quasi-finite maps from
	$\PGL_2$-torsors over an étale morphisms to the moduli space of Artin
	invariant $s$ surfaces. The latter space has dimension $s-1$, so each torsor
	has dimension $s+2$. By Chevalley's theorem, the image contains a locally
	closed subspace, which is necessarily of dimension $s+2$, as claimed.
\end{proof}

\subsection{Conditions imposed on $\WeierstrassData^n$ by Kodaira fiber types}
\label{sec:fiber moduli}

To a minimal Weierstrass datum $(\ms O_{\P^1}(n),a,b)$ we can associate a
minimal resolution of the associated Weierstrass fibration, yielding a smooth
relatively minimal elliptic surface $$X(\ms O(n),a,b)\to\P^1.$$ The singular
fibers of such a surface were described by Kodaira in characteristic $0$ and
Tate in positive characteristic \cite{MR0393039}. In this section, we consider
the conditions imposed on $\WeierstrassData^n$ by the presence of specific
singular fibers. We fix a negative integer $n$ and write simply
$\WeierstrassData$ instead of $\WeierstrassData^n$.

In the following, we will write $\Phi$ for a single Kodaira fiber type and $\bPhi$ for
an unordered list of fiber types (possibly with multiplicities). We focus on
fibers of additive type in the Weierstrass fibration. 

\begin{notation}
  Given a list of additive Kodaira fiber types $\bPhi$, we will say that a minimal
  Weierstrass datum $d:=(\ms O(n),a,b)$ \emph{realizes $\bPhi$\/} if for each additive Kodaira fiber type $\Phi$ appearing in $\bPhi$ with nonzero multiplicity $k$, the associated minimal elliptic surface $X(d)\to\P^1$ has precisely $k$ fibers of
  type $\Phi$ (no restriction is imposed on the semistable fibers of $X(d)$). We will write $\WeierstrassData[\bPhi]$ for the locus of Weierstrass data
  realizing $\bPhi$. 
\end{notation}

We can write the basic data $\alpha,\beta,\delta$ sheaf-theoretically for a Weierstrass datum $(\ms
O(n),a,b)$ in terms of the multiplicity of the divisors associated to the
sections $a\in\Gamma(\P^1,\ms O(-2n))$, $b\in\Gamma(\P^1,\ms O(-3n))$, and
$\Delta\in\Gamma(\P^1,\ms O(-6n))$.

\begin{lem}
	The locus $\WeierstrassData[\bPhi]\subset\WeierstrassData$ is locally closed.
\end{lem}
\begin{proof}
	% Some authors would "leave this as an exercise to the reader"

	It suffices to prove this statement for $\Poly[\bPhi]\subset\Poly$. There is a
	universal Weierstrass datum $(\ms O(n), a, b)$ over $\Poly$. Let
	$\bPhi=(\Phi_1,\ldots,\Phi_m)$. Let $\Delta\in\Gamma(\P^1\times\Poly,\ms
	O(-6n))$ be the discriminant, and consider the scheme $Z=Z(a)\cap Z(b)\subset\P^1\times\Poly$. The natural morphism $Z\to\Poly$ is finite
	but not necessarily flat. There is an open locus $U\subset\Poly$ over which
	the geometric fibers of $Z$ have precisely $m$ connected components,
	corresponding to precisely $m$ additive fibers (because the coefficients $a$ and $b$ of the Weierstrass equation have a common zero at precisely those points). The étale-local structure
	theory for finite morphisms tells us that there is a factorization $Z_U\to
	Y\to U$ with $Y\to U$ finite étale of degree $m$ and $Z_U\to Y$ radicial.
	Replacing $U$ with $\Isom(Y,\sqcup_{i=1}^{m} U)$ gives an $S_m$-torsor $I\to
	U$ over which $Y$ splits. Replacing $U$ with $I$, we may assume that there are
	sections $c_1,\ldots,c_m\in\P^1(U)$ that support the $m$ additive fibers.

	Choose a bijection $\sigma\in S_m$. For each fiber type $\Phi_i$ appearing in $\underline{\Phi}$, the condition that $a_{\sigma(i)}$
	supports $\Phi_i$ is equivalent to a certain requirement on the orders of vanishing of $a$, $b$, and $\Delta$ at $c_i$, as determined by Table \ref{table:1}. 
	%(For some fiber types, this sets a lower bound on the vanishing of $a$ or $b$ at the
	%point in question, and we mean that the section vanishes to at least that order at that point).
	For each $\Phi_i$ the resulting vanishing conditions determine a locally closed locus in $I$.
	Taking the union over all $\sigma\in S_m$ gives an $S_m$-invariant locally closed set
	of $I$ whose image in $\Poly$ is precisely $\Poly[\bPhi]$. This shows that
	$\Poly[\bPhi]$ is locally closed in $U$, as desired.
\end{proof}

The theory of Section \ref{sec:univ poly} gives a way to calculate the
codimension of $\WeierstrassData[\bPhi]$. Given an additive fiber configuration
$\bPhi$, let $\Poly[\bPhi]\subset\Poly$ denote the closed subspace over which
\eqref{eq:weier} has additive fiber configuration of type $\bPhi$, so that
$\Poly[\bPhi]$ maps to $\WeierstrassData[\bPhi]$.

\begin{lem}\label{lem:codim}
  For any additive fiber configuration $\bPhi$, we have
  $$\codim(\WeierstrassData[\bPhi]\subset\WeierstrassData)=\codim(\Poly[\bPhi]\subset\Poly).$$
\end{lem}
\begin{proof}
  This follows immediately from Lemma \ref{lem:coverings}.
\end{proof}

\begin{prop}\label{prop:fiber mass}
	For a list $\bPhi$ of additive fibers, the codimension of
  $\WeierstrassData^{ab\neq 0}[\bPhi]\subset\WeierstrassData^{ab\neq 0}$ is at
  least equal to $\zeta(\bPhi)$ (see Table \ref{table:1}).
\end{prop}
\begin{proof}
	It suffices to prove the corresponding statement for $\Poly^{ab\neq
	0}[\bPhi]\subset\Poly^{ab\neq 0}$. Consider the universal Weierstrass equation
	of type $\bPhi$, given as follows:
	$$y^2=x^3+a'_0\left(\prod_{i=1}^m(t-a'_i)^{\alpha_i}\prod_{j=m+1}^{-4n+m-\sum\alpha_i}(t-a'_j)\right)x+b'_0\prod_{i=1}^m(t-a'_i)^{\beta_i}\prod_{j=m+1}^{-6n+m-\sum\beta_i}(t-b'_j).$$
	Here the first $m$ roots of both coefficients are specified by the fiber types
	and the $\alpha$ and $\beta$ values of Table \ref{table:1}. The first factor
	in each coefficient is required to have distinct roots. The second factors in
	each coefficient are allowed to have roots coinciding with the first factor
	whenever the corresponding entry in Table \ref{table:1} has a $\geq\alpha_i$
	or $\geq\beta_i$ (depending upon the fiber type). 
	
	The locus $F(\bPhi)$ so described is open in the affine space
	$\A^{m-10n+2-\sum(\alpha_i+\beta_i)}$ with coordinates
	$$(a'_0,a'_1,\ldots,a'_{m-4n-\sum\alpha_i},
	b'_{m+1},\ldots,b'_{m-6n-\sum\beta_i}).$$ Choose a partition $\pi$ of the
	variables $(a_1,\ldots,a_{-4n})$ into $m$ subsets $A_i$, $i=1,\ldots,m$ of
	sizes $\alpha_i$, $i=1,\ldots,m$ and $-4n-\sum\alpha_i$ singletons $A'_j$, and
	a partition of $(b_1,\ldots,b_{-6n})$ into $m$ subsets $B_i$, $i=1,\ldots,m$
	of sizes $\beta_i$, $i=1,\ldots,m$ and $-6n-\sum\beta_i$ singletons $B'_j$.
	Given such a partition, we can define a closed immersion
	$$\iota_\pi:\A^{m-10n+2-\sum(\alpha_i+\beta_i)}\to\A^{-10n+2}_{a_0,\ldots,a_{-4n},b_0,\ldots,b_{-6n}}$$
	by the ring map sending all elements of $A_i$ and $B_i$ to $a'_i$, all of the
	singletons $A'_j$ to the $a'_j$, and all of the singletons $B'_j$ to $b'_j$.
	The union of the images of $F(\bPhi)$ under all of these maps is precisely
	$\Poly[\bPhi]$. The codimension is thus 
	$$-2+\sum(\alpha_i+\beta_i)=\sum(\alpha_i+\beta_i-1)\geq \sum \zeta(\Phi_i).$$ The
	last equality follows from the definition of $\zeta(\Phi)$, as
	one can see in Table \ref{table:1}.
\end{proof}

\section{The non $\infty$-Frobenius split locus in $\WeierstrassData^{-2}$}
\label{sec:non-pi-inj}

In this section, we establish estimates on the codimension of the supersingular
and non-$\infty$-Frobenius split loci in $\WeierstrassData^{-2}$. We split this
analysis into various pieces, which will be assembled in Proposition
\ref{prop:non-pi-inj-small-total} to obtain our final estimate.

\begin{notation}\label{notation:the number B of a configuration}
	Suppose that $\bPhi=(\Phi_1,\dots,\Phi_n)$ is a configuration of additive
	fibers. Let $\lambda_i=\delta(\Phi_i)$, and suppose that the $\Phi_i$ are ordered
	such that the sequence $(\lambda_1,\dots,\lambda_n)$ is non-increasing. We
	write $B=B_{13}(\lambda_1,\dots,\lambda_n)$ for the integer defined by
	(\ref{eq:the number B}).
\end{notation}

\begin{prop}\label{prop:codimension bounds using table}
   The locus in $\WeierstrassData^{-2,ab\neq 0}[\bPhi]$ parametrizing fibrations
   realizing $\bPhi$ that are not $\infty$-Frobenius split is a countable union
   of locally closed subspaces of codimension at least $B+\zeta(\bPhi)$ at every
   point.
\end{prop}
\begin{proof}
   Let $\Poly^{ab\neq 0}[\bPhi]$ be the preimage of $\WeierstrassData^{-2,ab\neq
   0}[\bPhi]$ under the flat cover of Lemma \ref{lem:coverings}. As in the proof
   of Proposition \ref{prop:fiber mass}, it is covered by copies of
   $\A^{22-\zeta(\bPhi)}$ with coordinates
   $$(a_0',a_1',\ldots,a'_{8+n-\sum\alpha_i},b_0',b'_1,\ldots,b'_{12+n-\sum\beta_i}).$$
   To such a point, we associate the polynomial
   $$
   \prod_{i=1}^n(t-a'_is)^{\lambda_i}.
   $$
   This is a projection onto the affine space $\A^n$ parametrizing polynomials
   of the form
   \[
     P(s,t)=\prod_{i=1}^n(t-z_is)^{\lambda_i}
   \]
   We can thus cover $\Poly^{ab\neq 0}[\bPhi]$ by components that are flat over
   $\A^n$. To prove the Proposition, it thus suffices to prove a corresponding
   codimension statement for certain subsets of $\A^n$.

   By Corollary \ref{cor:codim-case-13}, the locus in $\A^n$ parametrizing
   polynomials $P$ such that $P^{\frac{p^{2e}-1}{12}}$ is not $p^{2e}$-split for
   some $e$ is an ascending union of closed subschemes of codimension at least
   $B$. The preimage of this locus in $\Poly^{ab\neq 0}[\bPhi]$ is exactly equal
   to the preimage of the locus in $\WeierstrassData^{-2,ab\neq 0}[\bPhi]$
   parametrizing fibrations that are not $2e$-Frobenius split for some $e$. By
   Lemma \ref{lem:level lowering}, a fibration is not $2e$-Frobenius split for
   some $e$ if and only if it is not $\infty$-Frobenius split.
\end{proof}

Given a configuration $\bPhi$ of additive fibers, the integer $B+\zeta(\bPhi)$ can be easily
computed. We will show that, in almost all cases, this codimension is at least
13 (the remaining cases will be dealt with separately below). To reduce the
necessary computations, we first observe that certain configurations can never
contain non $\infty$-Frobenius split fibrations. For example, suppose that $f$ is a fibration in $\WeierstrassData^{-2}$ realizing $\bPhi$ such that $\deg(\FT(f))\leq 12$. For every $e>0$ we have
\[
\deg\left(\frac{p^{2e}-1}{12}\FT(f)\right)\leq p^{2e}-1
\]
By Proposition \ref{prop:frob-split}, this divisor is $p^{2e}$-split. Therefore,
by Proposition \ref{prop:cohoinj}(2) $f$ is $\infty$-Frobenius split. In fact,
something slightly stronger is true.
%Suppose that $\bPhi=(\Phi_1,\ldots,\Phi_n)$, and that $v(\Phi_1)$ is maximal among all of the $v(\Phi_i)$ (but there could be multiple fibers realizing the maximum).

\begin{lem}\label{lem:a little trick}
	Let $\bPhi$ be a configuration of additive fiber types and let $f$ be a fibration in $\WeierstrassData^{-2}$ realizing $\bPhi$. If there exists an $i$ such that $\delta(\bPhi)-\delta(\Phi_i)\leq 12$, then $f$ is $\infty$-Frobenius split.
\end{lem}
\begin{proof}
	For ease of notation, suppose that $i=1$. Let $p_1\in\P^1$ be the point supporting the fiber of type $\Phi_1$. Applying an automorphism of $\P^1$, we may assume that $p_1=[1:0]$, and that the fiber over $[0:1]$ is smooth. We then have that
	the divisor $\FT(f)$ corresponds to the polynomial
	\[
	    P(s,t)=t^{\delta(\Phi_1)}\prod_{i=2}^{n}(t-z_is)^{\delta(\Phi_i)}
	\]
	where the $z_i$ are non-zero. We see that the monomial with the highest power
	of $s$ appearing in the expansion of $P(s,t)$ is a scalar times
	$s^{\delta(\bPhi)-\delta(\Phi_1)}t^{\delta(\Phi_1)}$. It follows that the monomial with the
	highest power of $s$ appearing in the expansion of
	$P(s,t)^{\frac{p^{2e}-1}{12}}$ is the $\frac{p^{2e}-1}{12}$-th power of this
	term. Hence, $P(s,t)^{\frac{p^{2e}-1}{12}}$ contains a nonzero term of the
	form $s^\gamma t^\epsilon$ where $\gamma,\epsilon\leq p^{2e}-1$. As before, it
	follows from Proposition \ref{prop:frob-split} and Proposition
	\ref{prop:cohoinj} (2) that $f$ is $\infty$-Frobenius split.
\end{proof}

We make the following definition.
\begin{conditions}\label{conditions: fibers} Let $\bPhi$ be a configuration of additive fibers. Let $\Phi$ be a fiber type which achieves the maximal value of $\delta(\Phi)$ among all fiber types appearing in $\bPhi$. The fiber configuration $\bPhi$ is
	\emph{critical\/} if it satisfies the following conditions.
	\begin{enumerate}
		\item $\alpha(\bPhi)\leq 8$
		\item $\beta(\bPhi)\leq 12$
		\item $\delta(\bPhi) - \delta(\Phi) \geq 13$
	\end{enumerate}
\end{conditions}

The set of critical $\bPhi$ can be easily computed using Table \ref{table:1}. We
have included the resulting list in Table 2. In particular, we
notice that every critical configuration has non-$\infty$-Frobenius split locus
of codimension at least $13$.

\begin{prop}\label{prop:non split estimate ab not 0}
	The locus in $\WeierstrassData^{-2,ab\neq 0}$ parametrizing fibrations that
	are not $\infty$-Frobenius split is a countable union of locally closed
	subspaces of codimension at least $13$ at every point.
\end{prop}
\begin{proof}
	We will show that for each configuration $\bPhi$ of additive fiber types the
	non-$\infty$-Frobenius split locus in $\WeierstrassData^{-2,ab\neq 0}[\bPhi]$
	is a countable union of locally closed subspaces of codimension at least $13$
	at every point. This will prove the result.
	
	By Lemma \ref{lem:a little trick}, it suffices to consider configurations
	$\bPhi$ that are critical. Using Table \ref{table:2}, we see that for each
	such configuration, the codimension bound $B+\zeta(\bPhi)$ produced by Proposition
	\ref{prop:codimension bounds using table} is at least $13$. (To simplify this
	computation slightly, one could observe that, if $\zeta(\bPhi)\geq 13$, then by
	Proposition \ref{prop:fiber mass} the codimension of the entire locus of
	fibrations in $\WeierstrassData^{-2,ab\neq 0}$ realizing $\bPhi$ is already at
	least $13$. This leaves only 27 configurations in Table \ref{table:2} for
	which we must compute the integer $B$.)
\end{proof}

\begin{prop}\label{prop:a=0 is small}
	The supersingular locus in $\WeierstrassData^{-2,a=0}$ obeys the following
	dimensional constraints.
	\begin{enumerate}
		\item If $p\equiv 1\pmod{3}$, then the supersingular locus in
		$\WeierstrassData^{-2,a=0}$ is empty.
		\item If $p\equiv 2\pmod{3}$, then the supersingular locus in
		$\WeierstrassData^{-2,a=0}$ is a closed subspace whose codimension in
		$\WeierstrassData^{-2}$ is at least $14$.
	\end{enumerate}
\end{prop}
\begin{proof}
  We recall that, given any K3 surface $X$, the image of the map
	$\Aut(X)\to\GL(\H^0(X,\Omega^2_X))$ is a finite cyclic group, whose order is
	referred to as the \textit{non-symplectic index} of $X$. Consider a
	Weierstrass equation $y^2=x^3+b(t)$ parametrized by a point in
	$\WeierstrassData^{-2,a=0}$. The corresponding projective surface admits an
	action of $\mu_3$ where $\omega$ acts by $x\mapsto \omega x$ and $y\mapsto y$.
	The action extends to an action of $\mu_3$ on the associated K3 surface $X$.
	If $E_t\subset X$ is a smooth fiber of the elliptic fibration on $X$, then
	restriction induces an isomorphism
	$\H^0(X,\omega_X)\xrightarrow{\sim}\H^0(E,\omega_E)$, and the latter
	cohomology group is generated by the invariant differential $dx/2y$ (in
	Weierstrass coordinates). Thus, $\mu_3$ acts nontrivially on
	$\H^0(X,\Omega^2_X)$, and therefore the non-symplectic index of $X$ is
	divisible by $3$.
	
	The possible non-symplectic indices of supersingular K3 surfaces have been
	investigated by Jang in \cite{jang2019}. Table 1 of \cite{jang2019} shows that
	if $p\equiv 1\pmod{3}$, then no supersingular K3 surface admits a
	non-symplectic automorphism of order $3$, and if $p\equiv 2\pmod 3$, then the
	locus in $S_{10}$ consisting of supersingular K3
	surfaces with a non-symplectic automorphism of order 3 has dimension at most
	$4$. Here, $S_{10}$ is Ogus's moduli space parametrizing supersingular K3 surfaces $X$ together with a marking of $\Pic(X)$ by the supersingular K3 lattice of Artin invariant 10. The result now follows from Lemma \ref{lem:weierstrassify}.
\end{proof}

\begin{remark}
	In fact, Jang's results show more: the locus of points in $S_{10}$ with a
	non-symplectic automorphism of order $3$ cannot contain any point of even
	Artin invariant, and for each $0\leq n\leq 4$, has intersection with the Artin
	invariant $2n+1$ locus of dimension at most $n$. 
\end{remark}

\begin{prop}\label{prop:non-pi-inj-small-total}
	The locus $N\subset\WeierstrassData^{-2}$ parametrizing fibrations
	$(f:X\to\P^1,\sigma)$ such that $X$ is supersingular and $f$ is not
	$\infty$-Frobenius-split is a countable union of locally closed subspaces of
	codimension at least $13$ at every point.
\end{prop}
\begin{proof}
	It suffices to show the result separately for the corresponding loci in the
	space of Weierstrass data with $ab\neq 0$, $a=0$, and $b=0$.
	
	The first case follows from Proposition \ref{prop:non split estimate ab not
	0}, and the second follows from Proposition \ref{prop:a=0 is small}. Finally,
	consider $\Poly^{b=0}$. This subspace itself has codimension $13$, so $N\cap
	\WeierstrassData^{-2,b=0}$ has codimension at least $13$, as well.
\end{proof}

\begin{remark}
	In the case $ab\neq 0$, it seems quite difficult to obtain much information on
	the supersingular locus, at least in this generality. Fortunately, Proposition
	\ref{prop:non split estimate ab not 0} shows that the Frobenius nonsplit locus
	is already too small to contain general supersingular K3 surfaces. In the case
	$a=0$, the reverse is true: the condition of being not $\infty$-Frobenius
	split imposes only one relation, and hence the nonsplit locus could a priori
	contain a general supersingular K3 surface. However, Proposition \ref{prop:a=0
	is small} shows that in fact the supersingular locus is quite small.
\end{remark}

\section{Very general elliptic supersingular K3 surfaces}\label{sec:very general}

In this section we show that for very general supersingular K3 surfaces there
are no non-Jacobian elliptic structures that admit purely inseparable
multisections. We proceed by first studying Jacobian fibrations and then using
supersingular twistor lines to deduce consequences for non-Jacobian fibrations.

\begin{thm}\label{thm:hulk smush}
	Every Jacobian elliptic structure on a very general supersingular K3 surface
	of Artin invariant at least 7 is $\infty$-Frobenius-split.
\end{thm}
\begin{proof}
	With the notation introduced in the paragraph preceding Corollary \ref{cor:artin invariant weierstrassify}, we consider the countable collection of quasi-finite morphisms $w:M_i'\to\WeierstrassData^{-2}$. Each $M_i'$ is a $\PGL_2$-torsor over $V_i$, which is an open subset of the space of quadruples $((X,\tau),a,b)$ where $X$ is a supersingular K3 surface, $\tau$ is a lattice polarization of $X$ by a fixed supersingular K3 lattice, and $a,b$ are the divisor classes of a fiber or section respectively of a Jacobian elliptic fibration on $X$.
	
	By Corollary \ref{cor:dim cram}, the Artin invariant $s$ locus in $\WeierstrassData^{-2}$ has dimension at least $s+2$ at every point. By Proposition \ref{prop:non-pi-inj-small-total}, a very general point of the Artin invariant $s$ locus in $\WeierstrassData^{-2}$ is
	$\infty$-Frobenius-split for any $s$ at least $7$. The same follows for each $V_i$. As there are only countably many $V_i$, the very general point of Artin invariant at least $7$ in the Ogus moduli space is not in the image of the non $\infty$-Frobenius-split locus in any of the $V_i$. In other words, every Jacobian elliptic structure on a very general supersingular K3 surface of Artin invariant at least $7$ is $\infty$-Frobenius-split.
% 	There are countably many $V_i$, and 
% 	of Artin invariant $7$ surfaces with hyperbolic-plane polarized lattices
% 	parametrizes elliptic structures that are $\infty$-Frobenius-split. Since there
% 	are countably many hyperbolic-plane polarizations, we see that for a very general K3
% 	surface of Artin invariant $s$ for $s\geq 7$, every Jacobian elliptic
% 	structure is $\infty$-Frobenius-split. This concludes the proof.
\end{proof}

In order to explain consequences of Theorem \ref{thm:hulk smush}, we briefly
discuss twistor lines through generic surfaces. Given a positive integer $i$
between $1$ and $10$, let $\Lambda_i$ denote the supersingular K3 lattice of
Artin invariant $i$. We will write $M_i$ for the Ogus moduli space of
characteristic subspaces which parametrizes supersingular K3 crystals $H$
together with an isometric embedding $\Lambda_i\to H$. As Ogus describes, $M_i$
is smooth over $\F_p$ and irreducible, with $\Gamma(M_i,\ms O)=\F_{p^2}$. Let
$S_i$ denote the algebraic space parametrizing lattice-polarized supersingular
K3 surfaces $(X,\Lambda_i\inj\NS(X))$. Furthermore, there is
an étale morphism $S_i\to M_i$ and a covering family of Zariski opens $U\subset
M_i$ that admit sections $U\to S_i$ over $M_i$ (see the proof of \cite[Proposition 1.16]{Ogus83}). Write $\kappa_i$ for the
function field of $M_i$.

\begin{notation}\label{notn:burf}
	Suppose $\eta\to M_i$ is a geometric generic point, and choose a marked
	supersingular K3 surface $(X,\tau:\Lambda_i\simto\NS(X))$ realizing $\eta$.
\end{notation}

\begin{lem}\label{lem:generic is generic}
	Suppose given any class $a\in\Lambda_i$ that represents an elliptic fibration
	$X\to\P^1$. For any marking $\tau':\Lambda_{i-1}\simto\NS(J)$ on the Jacobian
	surface $J\to\P^1$ the resulting map $\eta\to M_{i-1}$ is a geometric generic
	point. 
\end{lem}
\begin{remark}
Note that a marking $\tau'$ as in Lemma \ref{lem:generic is generic} can only exist in the case where the fibration is not itself Jacobian, since $X$ must have Artin invariant $i$.
\end{remark}
\begin{proof}[Proof of \ref{lem:generic is generic}]
	Suppose the image of $\eta\to M_{i-1}$ has residue field $\kappa$. We can
	descend the pair $(J,\tau')$ to a pair defined over a finite extension
	$\kappa'\supset\kappa$.

	The Artin--Tate construction (see Section 4.3 of \cite{1804.07282}) produces a
	family
	$$\mf X\to\P^1_{\ms U_J}$$ where $\ms U_J$ is a scheme over $\kappa'$ that is
	a flat form of $\G_a$. There are two points $x_1,x_2:\eta\to\ms U_J$ and
	diagrams
	$$
	\begin{tikzcd}
		\mf X_{x_1}\ar[r]\ar[d] & J\ar[d] \\
		\P^1_\eta\ar[r] & \P^1_\eta
	\end{tikzcd}
	$$
	and
	$$
	\begin{tikzcd}
		\mf X_{x_2}\ar[r]\ar[d] & X\ar[d] \\
		\P^1_\eta\ar[r] & \P^1_\eta
	\end{tikzcd}
	$$
	in which all horizontal arrows are isomorphisms. (That is, the chosen
	fibration structures may only be defined over $\eta$, even though the
	cohomology classes are defined over $\kappa$.) There is an open subset
	$U\subset\ms U_J$ containing $x_2$ such that $\mf X_U\to U$ is a family of
	supersingular K3 surfaces of Artin invariant $i$. There is a global marking
	$m:\Lambda_i\simto\Pic_{\mf X_U/U}$. (We must therefore have that $m|_{x_2}$
	differs from $\tau$ by an isomorphism of $\Lambda_i$.) This gives a map $U\to
	M_i$. Since $x_2$ factors through $U$, we have that $U$ must hit the generic
	point of $M_i$. Since $\dim M_i=\dim M_{i-1}+1=\dim M_{i-1}+\dim U$, we have
	that $\kappa$ must have transcendence degree $i-1$ over $\F_p$, and thus
	$\eta\to\M_{i-1}$ hits the generic point.
\end{proof}

\begin{cor}\label{cor:artin tate all the things}
	Let $(X,\tau)$ be as in Notation \ref{notn:burf} and let
	$(Y,\tau':\Lambda_{i-1}\simto\NS(Y))$ be a marked supersingular K3 surface of
	Artin invariant $i-1$ over $\eta$ such that $\eta\to M_{i-1}$ is generic. If
	every Jacobian elliptic fibration $Y\to\P^1$ is $\infty$-Frobenius-split, then
	every elliptic fibration $X\to\P^1$ admitting a purely inseparable multsection
	is Jacobian.
\end{cor}
\begin{proof}
	Fix a non-Jacobian elliptic structure $f:X\to\P^1_\eta$. By \cite[Lemma 3.2.4]{1804.07282} and Lemma \ref{lem:generic is
	generic}, the Jacobian fibration $J\to\P^1_\eta$ is generic, so there is an
	isomorphism $J\simto Y$ of (unmarked) supersingular K3 surfaces. By
	assumption, then, the Jacobian fibration $J\to\P^1$ is
	$\infty$-Frobenius-split.

	Consider the Artin--Tate family $\mf X\to\P^1_{\A^1_{\eta}}$, and let
	$U\subset\A^1_{\eta}$ be the locus over which the fibers of $\mf X$ have Artin
	invariant $i$. There is a point $u\in U$ such that $\mf X_u\to\P^1_u$ is
	isomorphic to the chosen elliptic fibration $X\to\P^1$. The marking $\tau$ on
	$X$ extends to a marking $T:\Lambda_i\simto\Pic_{\mf X_U/U}$, inducing a
	diagram
	$$
	\begin{tikzcd}
		U\ar[r]\ar[d] & S_i\ar[d] \\
		\eta\ar[r]\ar[ur, dotted] & M_i.
	\end{tikzcd}
	$$
	The dotted arrow arises from the fact that $S_i\to M_i$ is étale,
	$\kappa(\eta)$ is algebraically closed, and $U$ is connected. This says
	precisely that there is an isomorphism
	\begin{equation}\label{eq:trivial}
		(\mf X_U,T)\simto(X,\tau)\times U.
	\end{equation}
	
	We now argue as in Corollary \ref{cor:they exist!}. Specifically, given a
	positive integer $e$ and a relative polarization $\ms O_{\mf X_U}(1)$ (over
	$U$), let $H_{e,b}\subset\Hom_U(\P^1,\mf X_U)$ be the open subscheme
	parametrizing morphisms whose composition $\P^1_U\to\mf X_U\to\P^1_U$ is the
	$e$th relative Frobenius and whose image has $\ms O(1)$-degree at most $b$ in
	each fiber. By Lemma \ref{lem:hom scheme thing}, $H_{e,b}\to U$ is of finite
	type. On the other hand, by Corollary \ref{cor:twistor}, the image of
	$H_{e,b}$ does not contain the generic point of $U$. It follows that the image
	is finite, hence that there is a point $t\to U(\eta)$ such that $\mf
	X_t\to\P^1_t$ does not admit a purely inseparable multisection of degree $p^e$
	over $\A^1$ and $\ms O(1)$-degree bounded by $b$. But by \eqref{eq:trivial} we
	then see that this is true of the chosen elliptic structure $f:X\to\P^1$.
	Applying this for all $e$ and $b$, we see that $f$ admits no purely
	inseparable multisection.
\end{proof}

\begin{cor}\label{cor:very general}
	Suppose $i$ is a positive integer between $8$ and $10$. If $X$ is a very
	general supersingular K3 surface of Artin invariant $i$, then every elliptic
	structure on $X$ admitting a purely inseparable multisection is Jacobian.
\end{cor}
\begin{proof}
	This follows from Theorem \ref{thm:hulk smush} and Corollary \ref{cor:artin
	tate all the things}.
\end{proof}

\begin{cor}\label{cor:hulk smash}
	No elliptic structure on a very general supersingular K3 surface admits a
	purely inseparable multisection.
\end{cor}
\begin{proof}
	This follows immediately from Corollary \ref{cor:very general}.
\end{proof}

\section{Table of critical configurations}\label{sec:tables}

Let $\bPhi=(\Phi_1,\dots,\Phi_n)$ be a configuration of additive fiber types. We
have defined two associated integers: $\zeta(\bPhi)$ (by definition,
$\zeta(\bPhi)=\sum_i\zeta(\Phi_i)$, where $\zeta(\Phi_i)$ is computed using Table
\ref{table:1}) and $B$ (defined in Notation \ref{notation:the number B of a
configuration}). By Proposition \ref{prop:fiber mass}, the locus
$\WeierstrassData^{-2,ab\neq 0}[\bPhi]$ of fibrations realizing $\bPhi$ has
codimension at least $\zeta(\bPhi)$. By Proposition \ref{prop:codimension bounds
using table}, the sublocus of $\WeierstrassData^{-2,ab\neq 0}[\bPhi]$
parametrizing fibrations that are not $\infty$-Frobenius split is a countable
union of locally closed subspaces of codimension at least $B+\zeta(\bPhi)$ at every
point.

The following table shows the full list of all additive fiber configurations
that are critical (in the sense of Condition \ref{conditions: fibers}), along
with the bounds $\zeta(\bPhi)$ and $B+\zeta(\bPhi)$. The key observation is that for
each critical fibration $\bPhi$ the bound $B+\zeta(\bPhi)$ for the codimension of
the non $\infty$-Frobenius split locus is at least $13$, so the full Artin
invariant $s$ locus cannot fit for $s$ at least $7$, as claimed in Proposition
\ref{prop:non split estimate ab not 0}.

\renewcommand{\arraystretch}{1.1}
\begin{center}\label{table:2}
  \begin{longtable}{|c|c|c|}
	\caption{Critical fiber configurations and associated bounds on
	codimensions}\\

	\hline 
	\multicolumn{1}{|c|}{Configuration} & \multicolumn{1}{c|}{$\zeta(\bPhi)$} &
		\multicolumn{1}{c|}{$B+\zeta(\bPhi)$} \\ 
		\hline 
		\endfirsthead
		
		\hline 
		\multicolumn{1}{|c|}{Configuration} & \multicolumn{1}{c|}{$\zeta(\bPhi)$}
		& \multicolumn{1}{c|}{$B+\zeta(\bPhi)$} \\ 
		
		\hline 
		\endhead
		
		\hline \multicolumn{3}{|r|}{{Continued on next page}} \\ 
		
		\hline
		\endfoot
		
		\hline
		\endlastfoot

		$8\II$ & 8 & 13\\
		\hline
		$5\II$ + $2\III$ & 9 & 13\\
		\hline
		$7\II$ + $\III$ & 9 & 14\\
		\hline
		$2\II$ + $4\III$ & 10 & 13\\
		\hline
		$5\II$ + $\III$ + $\IV$ & 10 & 14\\
		\hline
		$4\II$ + $3\III$ & 10 & 14\\
		\hline
		$6\II$ + $2\III$ & 10 & 14\\
		\hline
		$3\II$ + $\III$ + $2\IV$ & 11 & 14\\
		\hline
		$2\II$ + $3\III$ + $\IV$ & 11 & 14\\
		\hline
		$\II$ + $5\III$ & 11 & 14\\
		\hline
		$5\II$ + $\III$ + $\I_n^\ast$ & 11 & 14\\
		\hline
		$4\II$ + $2\III$ + $\IV$ & 11 & 14\\
		\hline
		$3\II$ + $4\III$ & 11 & 14\\
		\hline
		$5\II$ + $3\III$ & 11 & 14\\
		\hline
		$2\II$ + $\III$ + $2\I_n^\ast$ & 12 & 13\\
		\hline
		$\II$ + $\III$ + $3\IV$ & 12 & 14\\
		\hline
		$3\III$ + $2\IV$ & 12 & 14\\
		\hline
		$4\II$ + $2\I_n^\ast$ & 12 & 13\\
		\hline
		$3\II$ + $\III$ + $\IV$ + $\I_n^\ast$ & 12 & 14\\
		\hline
		$2\II$ + $3\III$ + $\I_n^\ast$ & 12 & 14\\
		\hline
		$2\II$ + $2\III$ + $2\IV$ & 12 & 14\\
		\hline
		$\II$ + $4\III$ + $\IV$ & 12 & 15\\
		\hline
		$6\III$ & 12 & 15\\
		\hline
		$4\II$ + $2\III$ + $\I_n^\ast$ & 12 & 14\\
		\hline
		$3\II$ + $3\III$ + $\IV$ & 12 & 15\\
		\hline
		$2\II$ + $5\III$ & 12 & 14\\
		\hline
		$4\II$ + $4\III$ & 12 & 14\\
		\hline
		$\II$ + $3\I_n^\ast$ & 13 & 14\\
		\hline
		$\III$ + $\IV$ + $2\I_n^\ast$ & 13 & 14\\
		\hline
		$2\II$ + $\IV$ + $2\I_n^\ast$ & 13 & 14\\
		\hline
		$\II$ + $2\III$ + $2\I_n^\ast$ & 13 & 14\\
		\hline
		$\II$ + $\III$ + $2\IV$ + $\I_n^\ast$ & 13 & 14\\
		\hline
		$3\III$ + $\IV$ + $\I_n^\ast$ & 13 & 15\\
		\hline
		$2\III$ + $3\IV$ & 13 & 15\\
		\hline
		$3\II$ + $\III$ + $2\I_n^\ast$ & 13 & 14\\
		\hline
		$2\II$ + $2\III$ + $\IV$ + $\I_n^\ast$ & 13 & 15\\
		\hline
		$\II$ + $4\III$ + $\I_n^\ast$ & 13 & 15\\
		\hline
		$\II$ + $3\III$ + $2\IV$ & 13 & 15\\
		\hline
		$5\III$ + $\IV$ & 13 & 16\\
		\hline
		$3\II$ + $3\III$ + $\I_n^\ast$ & 13 & 15\\
		\hline
		$2\II$ + $4\III$ + $\IV$ & 13 & 15\\
		\hline
		$\III$ + $3\I_n^\ast$ & 14 & 15\\
		\hline
		$2\IV$ + $2\I_n^\ast$ & 14 & 15\\
		\hline
		$2\II$ + $\III$ + $\I_n^\ast$ + $\IV^\ast$ & 14 & 14\\
		\hline
		$2\II$ + $3\I_n^\ast$ & 14 & 15\\
		\hline
		$\II$ + $\III$ + $\IV$ + $2\I_n^\ast$ & 14 & 15\\
		\hline
		$3\III$ + $2\I_n^\ast$ & 14 & 15\\
		\hline
		$2\III$ + $2\IV$ + $\I_n^\ast$ & 14 & 15\\
		\hline
		$2\II$ + $3\III$ + $\IV^\ast$ & 14 & 15\\
		\hline
		$2\II$ + $2\III$ + $2\I_n^\ast$ & 14 & 15\\
		\hline
		$\II$ + $3\III$ + $\IV$ + $\I_n^\ast$ & 14 & 16\\
		\hline
		$4\III$ + $2\IV$ & 14 & 16\\
		\hline
		$\II$ + $\III$ + $2\IV^\ast$ & 15 & 15\\
		\hline
		$\II$ + $2\I_n^\ast$ + $\IV^\ast$ & 15 & 15\\
		\hline
		$\III$ + $\IV$ + $\I_n^\ast$ + $\IV^\ast$ & 15 & 15\\
		\hline
		$\IV$ + $3\I_n^\ast$ & 15 & 16\\
		\hline
		$2\II$ + $\III$ + $\I_n^\ast$ + $\III^\ast$ & 15 & 15\\
		\hline
		$\II$ + $2\III$ + $\I_n^\ast$ + $\IV^\ast$ & 15 & 15\\
		\hline
		$\II$ + $\III$ + $3\I_n^\ast$ & 15 & 16\\
		\hline
		$3\III$ + $\IV$ + $\IV^\ast$ & 15 & 16\\
		\hline
		$2\III$ + $\IV$ + $2\I_n^\ast$ & 15 & 16\\
		\hline
		$\I_n^\ast$ + $2\IV^\ast$ & 16 & 16\\
		\hline
		$2\II$ + $2\III^\ast$ & 16 & 16\\
		\hline
		$\II$ + $\III$ + $\IV^\ast$ + $\III^\ast$ & 16 & 16\\
		\hline
		$\II$ + $2\I_n^\ast$ + $\III^\ast$ & 16 & 16\\
		\hline
		$2\III$ + $2\IV^\ast$ & 16 & 16\\
		\hline
		$\III$ + $\IV$ + $\I_n^\ast$ + $\III^\ast$ & 16 & 16\\
		\hline
		$\III$ + $2\I_n^\ast$ + $\IV^\ast$ & 16 & 16\\
		\hline
		$4\I_n^\ast$ & 16 & 17\\
		\hline
		$\IV$ + $2\III^\ast$ & 17 & 17\\
		\hline
		$\I_n^\ast$ + $\IV^\ast$ + $\III^\ast$ & 17 & 17\\
  \end{longtable}
\end{center}
%\printbibliography
\bibliography{biblio}{}
\bibliographystyle{hplain}
\end{document}